\documentclass[a4paper,10pt]{scrartcl}
\RequirePackage{enumerate}
\RequirePackage{verbatim}
\RequirePackage{hyperref}

\RequirePackage{amsmath,amsfonts,amssymb}
\RequirePackage[mathcal,mathbf]{euler}

\usepackage{latexsym,amscd,amsthm}
\usepackage[all]{xy}

\usepackage{amstext,amsopn}
\usepackage{graphicx}
\usepackage[dvips]{color}
\usepackage{epsfig}
\usepackage{psfrag}
\usepackage[all]{xy}
\usepackage{amscd}
\usepackage{amssymb}
\usepackage{fullpage}
\usepackage{url}

\usepackage{supertabular}

\usepackage{tikz}
\usetikzlibrary{arrows,decorations.markings}

\def\dar[#1]{\ar@<2pt>[#1]\ar@<-2pt>[#1]}
\entrymodifiers={!!<0pt,0.7ex>+}

\newtheorem{Thm}{Theorem}[section]
\newtheorem{Prop}[Thm]{Proposition}

\newtheorem{Claim}[Thm]{Claim}

\theoremstyle{remark}
\newtheorem{Rem}[Thm]{Remark}

\theoremstyle{definition}

\newtheorem{Def}[Thm]{Definition}
\newtheorem{Exa}[Thm]{Example}

\newcommand\kk{{\bold k}}
\newcommand\X{{\bold X}}
\newcommand\Y{{\bold Y}}
\newcommand\Map{{\bf Map}}

\newcommand\map{\mathbb{M}{\rm ap}}

\newcommand{\interval}{\begin{tikzpicture}
\node (A) at (0,0) {$\bullet$};
\node (B) at (0.5,0) {$\bullet$};
\draw [>=latex,-] (0,0) to (0.5,0);
\end{tikzpicture}}

\newcommand{\twopoints}{\begin{tikzpicture}
\node (A) at (0,0) {$\bullet$};
\node (B) at (0.3,0) {$\bullet$};
\end{tikzpicture}}

\author{Damien Calaque}
\title{Lagrangian structures on mapping stacks \\ and semi-classical TFTs}
\date{}
\begin{document}

\maketitle

\begin{abstract}
\noindent{\bf Abstract.} We extend a recent result of Pantev-To\"en-Vaqui\'e-Vezzosi \cite{PTVV}, 
who constructed shifted symplectic structures on derived mapping stacks having a Calabi-Yau source 
and a shifted symplectic target. Their construction gives a clear conceptual framework for 
the so-called AKSZ formalism \cite{AKSZ}.  

\medskip

We extend the PTVV construction to derived mapping stacks with boundary conditions, which is required 
in most applications to quantum field theories (see e.g.~the work of Cattaneo-Felder on the Poisson 
sigma model \cite{CF1,CF2}, and the recent work of Cattaneo-Mnev-Reshetikhin \cite{CMR1,CMR2}). 
We provide many examples of Lagrangian and symplectic structures that can be recovered in this way. 

\medskip

We finally give an application to topological field theories (TFTs). We expect that our approach will 
help to rigorously constuct a 2 dimensional TFT introduced by Moore and Tachikawa in \cite{MT}. 
A subsequent paper will be devoted to the construction of fully extended TFTs (in the sense of Baez-Dolan and Lurie \cite{BD,Lu1}) 
from mapping stacks. 
\end{abstract}

\setcounter{tocdepth}{2}
\tableofcontents

\section*{Introduction}
\addcontentsline{toc}{section}{Introduction}

\subsection*{Previous works}

\subsubsection*{AKSZ construction}

In the seminal paper \cite{AKSZ} the authors provide a general procedure that allows one to put many $\sigma$-models in the framework 
of the BV formalism \cite{BV}. In the recent papers \cite{CMR1,CMR2} this construction is extended to the case when the source of the 
$\sigma$-model has a boundary, and the authors expect that this can further be extended to manifolds with corners and ultimately lead, 
in the case of  topological $\sigma$-models, to fully extended topological field theories as they are defined in \cite{BD,Lu1}. 
The main problems of the AKSZ construction are: \\[-0.5cm]
\begin{enumerate}
\item it deals with mapping spaces which are infinite dimensional. \\[-0.6cm]
\item the mapping spaces considered are formal, in the sense that they actually only capture those maps which are infinitesimally 
close to the constant ones. \\[-0.5cm]
\end{enumerate}

\subsubsection*{Fully extended TFTs}

Usual topological field theories (TFTs) are axiomatized as symmetric monoidal functors from the category $nCob$ defined in the following way. Objects are closed manifolds of dimension $n-1$, morphisms are diffeomorphism classes of $n$-dimensional cobordisms between those, and the 
monoidal structure is given by the disjoint union (notice that there are several variants: unoriented, oriented, framed, $\dots$). Except 
for the cases $n=1$ and $n=2$ it is very difficult to describe and/or characterise $nCob$. There are two successive extensions of $nCob$ 
one can consider: 
\begin{enumerate}
\item in \cite{BD} the authors introduce an $n$-category $nCob^{ext}$ of fully extended cobordisms: objects are closed $0$-dimensional manifolds, 
$1$-morphisms are $1$-cobordisms between them, $2$-morphisms are $2$-dimensional cobordisms between the laters, $\dots$, and $n$-morphisms 
are diffeomorphism classes of cobordisms between $n-1$-dimensional ones. This allows to compute, at least theoretically, the invariant 
associated to a closed $n$-manifold (viewed as a cobordism between $\emptyset$ and itself) from a triangulation. 
\item in \cite{Lu1} it is argued that one should keep track of diffeomorphisms and it is proposed to consider an $(\infty,n)$-category 
$nCob^{ext}_\infty$ of up-and-down extended cobordisms. The meaning of $(\infty,n)$ here is that we have a higher category in which all 
$k$-morphisms are (weakly) invertible for $k>n$. It is very similar to $nCob^{ext}$, except that one keeps track of diffeomorphisms between 
$n$-dimensional cobordisms and homotopies between them: $(n+1)$-morphisms are diffeomorphisms, $(n+2)$-morphisms are isotopies between them, 
$(n+3)$-morphisms are isotopies between isotopies, $\dots$
\end{enumerate}
The main result of \cite{Lu1} is a characterisation of the framed version $nCob^{ext,fr}_\infty$ of $nCob^{ext}_\infty$ as the free symmetric monoidal $(\infty,n)$-category generated by a fully dualizable object (a notion we won't explain here, and that should be understood as a strong generalization of finite dimensionality for a vector space in the context of objects in a symmetric monoidal higher category). 

\subsubsection*{PTVV construction}

In a recent paper \cite{PTVV} the AKSZ construction has been interpreted and re-written in the realm of derived (algebraic) geometry 
(see e.g.~\cite{TV}). This new approach has several advantages: \\[-0.5cm]
\begin{enumerate}
\item there is no infinite dimensional complication. There are representability theorems by Lurie and To\"en-Vezzosi which guaranty 
that the mapping stacks we are going to consider are tractable. \\[-0.6cm]
\item mapping stacks are no longer formal. From the point-of-view of quantization it gives a hope that we will be able to produce 
non-perturbative quantum field theories using this formalism. \\[-0.6cm]
\item derived geometry is formulated in the language of homotopy theory, which is also the one of modern higher category theory, 
and is therefore {\it a priori} well-suited for fully extended TFTs. \\[-0.5cm]
\end{enumerate}
We briefly summarize the AKSZ/PTVV construction: if ${\bf X}$ is a (derived Artin) stack that is ``compact'' and admits a $d$-orientation 
$[{\bf X}]$, and if ${\bf Y}$ is a stack equipped with an $n$-symplectic structure $\omega$, then $\int_{[{\bf X}]}ev_{\bf X}^*\omega$, 
where $ev:{\bf X}\times\Map({\bf X},{\bf Y})\to {\bf Y}$ is the evaluation map and $\int_{[{\bf X}]}$ denotes the integration against the 
fundamental class $[{\bf X}]$, is an $(n-d)$-symplectic structure on the derived mapping stack $\Map({\bf X},{\bf Y})$. 
The main examples to keep in mind are when ${\bf Y}=BG$, which is $2$-symplectic whenever $G$ is a reductive algebraic group, and \\[-0.5cm]
\begin{enumerate}
\item ${\bf X}=\Sigma_B$ is the homotopy type of a compact oriented surface $\Sigma$: the $0$-symplectic structure on 
$\Map({\bf X},{\bf Y})={\bf Loc}_G(\Sigma)$ coincides with the genuine symplectic structure on the moduli space of $G$-local systems. \\[-0.6cm]
\item ${\bf X}$ is a $K3$ surface: the $0$-symplectic structure on $\Map({\bf X},{\bf Y})={\bf Bun}_G({\bf X})$ coincides with the 
genuine symplectic structure on the moduli space of $G$-bundles on ${\bf X}$ discovered by Mukai \cite{Mu}. 
\end{enumerate} 

\subsection*{Motivational conjectures and main results}

\subsubsection*{Classical fully extended TFTs from mapping stacks}

To any stack ${\bf X}$ we associate the $(\infty,0)$-category (or, $\infty$-groupoid) $Corr_{(\infty,0)}({\bf X})$ of stacks over $X$, 
where we have discarded all non-invertible morphisms. 
Assuming we have been able to construct an $(\infty,n)$-category $Corr_{(\infty,n)}({\bf X})$ for any stack ${\bf X}$, we then 
define $(\infty,n+1)$-categories $Corr_{(\infty,n+1)}({\bf X})$ with objects being stacks over $X$ and having 
$Corr_{(\infty,n)}({\bf Y}_1\times_{{\bf X}}^h {\bf Y}_2)$ as $(\infty,n)$-category of morphisms from ${\bf Y}_1\to {\bf X}$ to 
${\bf Y}_2\to {\bf X}$. If ${\bf X}=*$ then $Corr_{(\infty,n)}(*)$ is symmetric monoidal, with $\otimes:=\times$. 
We conjecture the following: 
\begin{quote}
{\it Any object ${\bf Y}$ in $Corr_{(\infty,n)}(*)$ is fully dualizable. The induced symmetric monoidal $(\infty,n)$-functor 
$\mathcal Z_{{\bf Y},\infty}:nCob_\infty^{ext,fr}\longrightarrow Corr_{(\infty,n)}(*)$ can be realized as $\Map\big((-)_B,{\bf Y}\big)$ and 
factors through $nCob_\infty^{ext}$ (roughly, it doesn't depend on framings). }
\end{quote}
In the present paper we only consider non fully extended TFTs. We show in Subsection \ref{ss-TFTmap} that $\Map\big((-)_B,{\bf Y}\big)$ 
produces an honest $n$-dimensional TFT 
$\mathcal Z_{{\bf Y}}:nCob\to Corr$, where $Corr$ is a genuine category of correspondences obtained by ``truncating'' 
$Corr_{(\infty,n)}(*)$ both above and below: objects in $Corr$ are the $(n-1)$-endomorphisms of the unit in $Corr_{(\infty,n)}(*)$ 
(namely, stacks), and morphisms are equivalence classes of $n$-morphisms (i.e.~correspondences). 

\subsubsection*{Semi-classical fully extended TFTs from mapping stacks}

We now put a bit of ``symplectic flavor'' in the above construction. 
Before doing so let us recall from \cite{PTVV} that there is a notion of Lagrangian structure for a map $f:{\bf Y}\to{\bf X}$ when 
${\bf X}$ is $n$-symplectic. One observes that an $n$-symplectic structure on a stack ${\bf X}$ is the same as a Lagrangian structure 
on the map ${\bf X}\to *_{(n+1)}$, where $*_{(n+1)}$ denotes the point with its canonical $(n+1)$-symplectic structure. 
Given a $k$-symplectic stack ${\bf X}$ we can define an $(\infty,0)$-category $Lag_{(\infty,0)}({\bf X})$ of maps 
${\bf Y}\to {\bf X}$ equipped with a Lagrangian structure. 
Assuming we have been able to construct an $(\infty,n)$-category $Lag_{(\infty,n)}({\bf X})$ for any $k$-symplectic stack ${\bf X}$ 
(and any $k$), we then define $(\infty,n+1)$-categories $Lag_{(\infty,n+1)}({\bf X})$ with objects being morphisms 
${\bf Y}\to{\bf X}$ equipped with a Lagrangian structure and having $Lag_{(\infty,n)}({\bf Y}_1\times_{\bf X}^h {\bf Y}_2)$ as 
$(\infty,n)$-category of morphisms from ${\bf Y}_1\to{\bf X}$ to ${\bf Y}_2\to{\bf X}$. 
Here we use the fact from \cite{PTVV} that the homotopy fiber product of two Lagrangian morphisms to a given $k$-symplectic stack is 
$(k-1)$-symplectic. As before, if ${\bf X}=*_{(k)}$ then $Lag_{(\infty,n)}(*_{(k)})$ is symmetric monoidal with $\otimes:=\times$. 
We conjecture the following: 
\begin{quote}
{\it Any object ${\bf Y}$ in $Lag_{(\infty,n)}(*_{(k)})$ is fully dualizable. 
The induced symmetric monoidal $(\infty,n)$-functor 
$\mathcal Z^{or}_{{\bf Y},\infty}:nCob_\infty^{ext,fr}\longrightarrow Lag_{(\infty,n)}(*_{(k)})$ can be realized as $\Map\big((-)_B,{\bf Y}\big)$ 
and factors through $nCob_\infty^{ext,or}$ (roughly, it only depends on the orientation). }
\end{quote}
Again, in this paper we only consider non fully extended TFTs. We show in Subsection \ref{ss-TFTlag} that $\Map\big((-)_B,{\bf Y}\big)$ produces 
an honest $n$-dimensional oriented TFT 
$\mathcal Z^{or}_{{\bf Y}}:nCob^{or}\to LagCorr_{n-k}$, where $LagCorr_{n-k}$ is a genuine category of Lagrangian correspondences between 
$(n-k)$-symplectic stacks obtained by ``truncating'' $Lag_{(\infty,n)}(*_{(k)})$ both above and below: objects in 
$LagCorr_{n-k}$ are the $(n-1)$-endomorphisms of the unit in $Lag_{(\infty,n)}(*_{k}))$ (namely, $(n-k)$-symplectic stacks), 
and morphisms are equivalence classes of $n$-morphisms (Lagrangian correspondences). 

Our proof mainly relies on the following result (see Theorem \ref{thm-maplag}): 
\begin{quote}
{\it If $\Sigma$ is a compact $(d+1)$-manifold with boundary $\partial \Sigma$, then the restriction morphism \linebreak
$\Map\big((\Sigma)_B,{\bf Y}\big)\longrightarrow\Map\big((\partial\Sigma)_B,{\bf Y}\big)$ has a natural Lagrangian structure. 
This is a consequence of a more general fact: if $f:{\bf X}\to {\bf X}'$ is a morphism between ``compact'' stacks together with 
a relative $d$-orientation then the morphism $\Map\big({\bf X}',{\bf Y}\big)\overset{f^*}{\longrightarrow}\Map\big({\bf X},{\bf Y}\big)$ 
is naturally equipped with a Lagrangian structure. }
\end{quote}
Let us give three examples of Lagrangian structures arrizing in this way: \\[-0.5cm]
\begin{enumerate}
\item if $\Sigma=\raisebox{-0.5ex}{\interval}$, with $\partial\Sigma=\raisebox{-0.5ex}{\twopoints}$, then $\mathcal Z_{\bf Y}^{or}(\Sigma)$ 
is the diagonal map ${\bf Y}\to {\bf Y}\times \overline{{\bf Y}}$ equipped with its standard Lagragian structure (here $\overline{{\bf Y}}$ 
means that we equip ${\bf Y}$ with the opposite shifted symplectic structure). \\[-0.6cm] 
\item if $dim(\Sigma)=l$ and $\partial\Sigma=\emptyset$ then $\mathcal Z^{or}_{\bf Y}(\Sigma)$ is the morphism 
$\Map\big((\Sigma)_B,{\bf Y}\big)\to*_{(k-l)}$ and the Lagrangian structure on it is precisely the $(k-l-1)$-symplectic structure 
on the mapping stack obtained {\it via} PTTV. \\[-0.6cm]
\item if ${\bf X}'$ is a Fano three-fold with smooth anticanonical divisor ${\bf X}$ and ${\bf Y}=BG$, $G$ being a reductive algebraic group, 
then we recover that ${\bf Bun}_G({\bf X}')\to{\bf Bun}_G({\bf X})$ has a Lagrangian structure (see e.g. \cite{T}). \\[-0.5cm]
\end{enumerate}
Before going further and describe the contents of the present paper, let us mention three classical field theories 
that one can recover in this way: \\[-0.5cm]
\begin{enumerate}
\item if $n=k=3$ and ${\bf Y}=BG$ then we get the classical Chern-Simons theory. \\[-0.6cm]
\item if $n=3$, $k=1$ and ${\bf Y}$ is a genuine smooth symplectic variety, then we get classical Rozansky-Witten theory \cite{RW}. \\[-0.6cm]
\item {\it Conjectural. }In \cite{MT} the authors define a TFT taking values in a category of symplectic varieties 
equipped with a  Hamiltonian action, and urge mathematicians to construct it rigorously. 
Our approach seems very well-adapted to attack such a problem. \\[-0.5cm]
\end{enumerate}

\subsection*{Description of the paper}

We start in Section 1 with some recollection from \cite{PTVV} on $n$-symplectic structures. 
We continue in Section 2 with some recollection on Lagrangian structures, and provide new examples of these. 
In particular, we state and prove our main Theorem on Lagrangian structures on derived mapping stacks (Theorem \ref{thm-maplag}). 
We briefly explain in Section 3 how can one recover already known symplectic and Lagrangian moduli spaces (such as the symplectic structure on 
$G$-local systems on punctured surfaces with prescribed conjugacy classes of monodromy around punctures). 
We prove in Section 4 that mapping stacks with Betti source and $n$-symplectic target define semi-classical topological field theories 
(semi-classical meaning that they take values in some category of Lagrangian correspondences). We finally conclude the paper with some perspectives 
and a short discussion of boundary conditions. 

\subsubsection*{Acknowledgements} 

I thank Bertrand T\"oen for many very helpful discussions, Michel Vaqui\'e for his kindness in answering a few basic questions, 
and Tony Pantev for pointing \cite{MT} to me. I also thank Pavel Safronov and the anonymous referee, who pointed a few inaccuracies 
in a previous version. 

I started this project after having heard a very enlightening talk by Alberto Cattaneo. 
The present paper can be seen as a reformulation of some part of \cite{CMR1,CMR2} in a totally different language\footnote{We should 
emphasize that these references actually deal with more general theories than the ones of AKSZ type. }. 
My subsequent work will be an attempt to make precise \cite[Remark 4.2]{CMR1} and \cite[\S5.3]{CMR2} using this language. 

This work has been partially supported by a grant from the Swiss National Science Foundation 
(project number $200021\underline{~}137778$).

\subsubsection*{Notation}

Below are the notation and conventions we use in this paper. They can easily be skipped. 
\begin{itemize}
%\item all along we fix three universes $\mathbb{U}\in\mathbb{V}\in\mathbb{W}$. As a matter of simplification, we use the following convention: 
%$\mathbb{U}$-small sets are called {\it sets}, $\mathbb{V}$-small sets are called {\it big sets}, $\mathbb{W}$-small sets are called 
%{\it huge sets}, $\mathbb{U}$-small categories are called {\it small categories}, $\mathbb{V}$-small categories 
%are called {\it categories}, and $\mathbb{W}$-small categories are called {\it big categories}. E.g.~a big category has a huge set of objects. 
\item our models for $(\infty,1)$-categories are categories with weak equivalences (a-k-a relative categories). 
We refer to \cite{BK} for the details about the homotopy theory of relative categories. 
\item we write $h\mathcal C$ for the homotopy category of an $(\infty,1)$-category $\mathcal C$. 
\item there is a notion of weak equivalence between relative categories. Hence we have an $(\infty,1)$-category of 
(small\footnote{We will mainly ignore size issues. }) $(\infty,1)$-categories. 
%\item similarly, we have a big $(\infty,1)$-category $CAT_\infty$ of $(\infty,1)$-categories. More generally we will adopt the following 
%convention: if $Xyz$ denotes a category of which the construction makes sense if we replace $(\mathbb{U},\mathbb{V})$ by $(\mathbb{V},\mathbb{W})$, 
%then we will denote by $XYZ$ the corresponding big category. 
\item $Top$ is a good category of topological spaces, and $sSet$ is the category of simplicial sets. They are weakly equivalent as 
relative categories. The corresponding $(\infty,1)$-category will be called {\it the $(\infty,1)$-category of spaces}. 
\item if $\mathcal C$ is an $(\infty,1)$-category then we write $\map_{\mathcal C}(x,y)$ for the {\it space} of morphisms 
from an object $x$ to another object $y$ in $\mathcal C$. Paths in $\map_{\mathcal C}(x,y)$ will be refered to as 
{\it homotopies} between morphisms. 
\item a morphism in an $(\infty,1)$-category is called an {\it equivalence} if it induces an isomorphism in the homotopy category 
(i.e.~if it is homotopic to a zig-zag of weak equivalences). 
\item $\Gamma$ is the category of pointed sets of the form $[n]:=\{*,1,\dots,n\}$, with morphisms being pointed maps. 
Given a pointed set $X$ and $x\in X\backslash\{*\}$ we denote by $(x):X\to[1]:=\{*,1\}$ the map that sends $x$ to $1$ and all the other 
elements to $*$. More generally we denote a pointed map $f:[n]\to [m]$ as an $n$-tuple $\big(f^{-1}(1),\dots,f^{-1}(m)\big)$ of disjoint 
subsets of $[m]$. 
\item a {\it symmetric monoidal $(\infty,1)$-category} is a $\Gamma$-shaped diagram $\mathcal C$ of relative categories 
such that for any object $X$ of $\Gamma$ the product $\prod_{x\in X}\mathcal C_{(x)}:\mathcal C_X\longrightarrow(\mathcal C_{[1]})^X$ 
is a weak equivalence. 
\item for $n$ objects $x_1,\dots,x_n$ of a symmetric monoidal $(\infty,1)$-category $\mathcal C$ (by which we mean objects of 
$\mathcal C_{[1]}$) we write $x_1\otimes \cdots\otimes x_n$ for the image through $(1\cdots n)$ of any object $u$ in $\mathcal C_{[n]}$ 
such that $\mathcal C_{(i)}(u)=x_i$ for all $i$. This is only defined up to a weak equivalence. 
We also write $1_{\mathcal C}$ for the image of the unique object {\it via} $\mathcal C_{[0]}\to \mathcal C_{[1]}$. 
\item an object $x$ is called {\it dualizable} if there exists an object $x^\vee$ and morphisms $1_{\mathcal C}\to x\otimes x^\vee$ and 
$x^\vee\otimes x\to 1_{\mathcal C}$ such that the composed maps 
$x^\vee\to x^\vee\otimes (x\otimes x^\vee)\tilde\to (x^\vee\otimes x)\otimes x^\vee\to x^\vee$ 
and $x\to (x\otimes x^\vee)\otimes x\tilde\to x\otimes (x^\vee\otimes x)\to x$ are homotopic to identities. 
\item in order to remain on the safe side $\kk$ is a field of characteristic zero (but we can probably allow it to be a commutative 
ring which is Noetherian and of residual characteristic zero). 
\item for a dualizable $\kk$-module $V$ we might denote $V^*$ its dual $V^\vee$. 
\item $Cpx$ is the symmetric monoidal $(\infty,1)$-category of cochain complexes of $\kk$-modules. 
\item $dSt_{{\bf k}}$ is the symmetric monoidal $(\infty,1)$-category of derived stacks over $\kk$ for the \'etale topology. 
The symmetric monoidal structure is closed: for any two stacks ${\bf X},{\bf Y}$ there is a ``mapping stack'' ${\bf Map}({\bf X},{\bf Y})$. 
Given a derived stack ${\bf X}$ we also have the $(\infty,1)$-category $dSt_{/{\bf X}}$ of derived stacks over ${\bf X}$. 
\item if ${\bf X}$ is a derived stack then $QCoh({\bf X})$ is the symmetric monoidal $(\infty,1)$-category of quasi-coherent sheaves on ${\bf X}$. 
We also consider the symmetric monoidal $(\infty,1)$-category $Sh({\bf X})$ of $\kk_{\bf X}$-modules. 
\item a {\it derived Artin stack} is a derived stack which is $m$-geometric (for some $m$) w.r.t.~the class of smooth morphisms 
(see \cite{HAG-II}) and which is locally of finite presentation (this is a bit more restrictive than the usual Artin condition). 
In particular, any derived Artin stack $X$ has a dualizable cotangent complex $\mathbb{L}_X$, and thus one can define its tangent 
complex $\mathbb{T}_X:=\mathbb{L}_X^\vee$. 
\end{itemize}

\section{Recollection on shifted symplectic structures}

\subsection{Definitions}

In this Subsection we summarize and follow closely \cite[Section 1]{PTVV}, to which we refer for the details. 
All along, $\X$ will be a derived Artin stack. We also provide two new examples of $1$-symplectic structures, 
appearing naturally in Lie theory. 

\subsubsection{$p$-forms and closed $p$-forms}

Let us consider the quasi-coherent weighted sheaf $\Omega_\X:=S_{\mathcal O_\X}(\mathbb{L}_\X[1])$, 
where the weight $p$ subsheaf is $\Omega^p_\X:=S^p_{\mathcal O_\X}(\mathbb{L}_\X[1])$. 
The space of {\it $p$-forms of degree} $n$ is 
$$
\mathcal A^p(\X,n):=\mathbb{M}\mathrm{ap}_{Sh(\X)}(\kk_\X,\Omega_\X^p[n-p])\cong
\mathbb{M}\mathrm{ap}_{QCoh(\X)}(\mathcal O_\X,\Omega_\X^p[n-p])\,.
$$
\begin{Rem}\label{Rem-ev}
$\Omega_\X$ can be identified with the sheaf $ev(\bullet)_*\mathcal O_{{\bf Map}(B\mathbb{G}_a,\X)}$, where 
$ev(\bullet):{\bf Map}(B\mathbb{G}_a,\X)\to \X$ is the evaluation at the canonical point $\bullet\to B\mathbb{G}_a$. 
The weight is given by the $\mathbb{G}_m$-action on $B\mathbb{G}_a$. 
\end{Rem}
From the above Remark we see that $\Omega_\X$ inherits a weight $1$ action of $B\mathbb{G}_a$. 
In concrete terms it boils down to the action of the de Rham differential, which extends the derivation 
$\mathcal O_\X\to\mathbb{L}_\X=\mathbb{L}_\X[1][-1]$. 

We denote by $\Omega^{cl}_\X$ the $B\mathbb{G}_a$-homotopy fixed points of $\Omega_\X$ within $\mathbb{G}_m\textrm{-}Sh(\X)$, 
and $\Omega^{p,cl}_\X$ its weight $p$ subsheaf. 
The space $\mathcal A^{p,cl}(\X,n)$ of {\it closed $p$-forms of degree} $n$ is then 
$\mathbb{M}\mathrm{ap}_{Sh(\X)}(\kk_\X,\Omega_\X^{p,cl}[n-p])$. 

There is an obvious morphism $\mathcal A^{p,cl}(\X,n)\to\mathcal A^{p}(\X,n)$. The image, under this map, of a closed $p$-form of degree $n$ 
is called its {\it underlying $p$-form}. Given a $p$-form $\omega\in \mathcal A^{p}(\X,n)$ of degree $n$ we will call its homotopy fiber 
$\mathcal A^{p,cl}(\X,n)\underset{\mathcal A^p(\X,n)}{\overset{h}{\times}}\{\omega\}$ the {\it space of keys} of $\omega$. 

\subsubsection{Shifted sympletic structures}

A $2$-form $\omega:\mathcal O_\X\to S^2_{\mathcal O_\X}(\mathbb{L}_\X[1])[n-2]$ of degree $n$ is said {\it non-degenerate} 
if the induced map $\mathbb{T}_\X\to \mathbb{L}_\X[n]$ is an isomorphism in $hQCoh(\X)$. We denote by $\mathcal A^2(\X,n)^{nd}$ 
the subspace of non-degenerate $2$-forms of degree $n$. 
\begin{Rem}
$\mathcal A^2(\X,n)^{nd}$ is a union of connected components in $\mathcal A^2(\X,n)$, therefore the map 
$\mathcal A^2(\X,n)^{nd}\to\mathcal A^2(\X,n)$ is a fibration. 
\end{Rem}
An {\it $n$-symplectic form} is a closed $2$-form of which the underlying $2$-form is non-degenerate. 
We write $\mathbb{S}\mathrm{ymp}(\X,n)$ for the space of $n$-symplectic forms, defined as the (homotopy) pull-back
$$
\mathbb{S}\mathrm{ymp}(\X,n)=\mathcal A^2(\X,n)^{nd}\!\!\!\!\underset{\mathcal A^2(\X,n)}{\times}\!\!\!\!\mathcal A^{2,cl}(\X,n)
=\mathcal A^2(\X,n)^{nd}\!\!\!\!\overset{h}{\underset{\mathcal A^2(\X,n)}{\times}}\!\!\!\!\mathcal A^{2,cl}(\X,n)\,.
$$

\subsection{Examples of shifted symplectic structures}

In this Subsection we recall the main results from \cite[Section 2]{PTVV}. 

\subsubsection{$2$-symplectic structures on $BG$}

Let $G$ be a group scheme and let $\X=BG$. Then $\mathbb{L}_\X=\mathfrak{g}^*[-1]$ and thus $\Omega_\X=S(\mathfrak g^*)$. 
Since $\Omega_\X$ is concentrated in (cohomological) degree $0$ then the action of $B\mathbb{G}_a$ on it is trivial. 
\begin{Rem}
It is a general fact that the homotopy fixed points of the trivial action of $B\mathbb{G}_a$ on a weighted object $C$ are given by 
$\bigoplus_{i\geq0}C[-2i](i)=C[u]$, where $u$ is a degre $2$ variable of weight $-1$. 
\end{Rem}
Let us now assume that $G$ is reductive. We then have: 
\begin{itemize}
\item $\mathcal A^{p}(\X,n)=\mathbb{M}\mathrm{ap}_{G\textrm{-}mod}\big(\kk,S^p(\mathfrak{g}^*)[n-p]\big)
=\mathbb{M}\mathrm{ap}_{Cpx}\big(\kk,S^p(\mathfrak{g}^*)^G[n-p]\big)$; 
\item $\mathcal A^{p,cl}(\X,n)=\mathbb{M}\mathrm{ap}_{Cpx}\left(\kk,\bigoplus_{i\geq0}S^{p+i}(\mathfrak{g}^*)^G[n-p-2i]\right)$. 
\end{itemize}
It appears clearly that a $2$-form can only be non-degenerate if its degree is $2$. Therefore we have 
$$
\mathbb{S}\mathrm{ymp}(\X,n)_0=\mathbb{S}\mathrm{ymp}(\X,2)_0=\pi_0\big(\mathbb{S}\mathrm{ymp}(\X,2)\big)
=\{\textrm{invariant non-degenerate symmetric bilinear forms on }\mathfrak g\}\,.
$$

\subsubsection{A $2$-symplectic structure on ${\bf Perf}$}

Recall that for any perfect complex $E$ over a stack $\X$ its pull-back through the evaluation map 
$B\mathbb{G}_a\times{\bf Map}(B\mathbb{G}_a,\X)\to \X$ is a $B\mathbb{G}_a$-equivariant perfect complex on 
$B\mathbb{G}_a\times{\bf Map}(B\mathbb{G}_a,\X)$, which we can view as a $B\mathbb{G}_a$-equivariant perfect complex $F$ on 
${\bf Map}(B\mathbb{G}_a,\X)$ together with a $B\mathbb{G}_a$-equivariant $u:F\to F$. The trace of $u$ defines a 
$B\mathbb{G}_a$-equivariant function $Ch(E)$ on ${\bf Map}(B\mathbb{G}_a,\X)$, of degree $0$. Its homogeneous part of 
weight $p$ then defines a closed $p$-form of degree $p$, denoted $Ch(E)_p$, on $\X$. 

\medskip

Applying the above scheme to the tautological perfect complex $\mathcal E$ on the stack ${\bf Perf}$ of perfect complexes, 
as defined in \cite[Definition 1.3.7.5]{HAG-II}, we get for $p=2$ a closed $2$-form $Ch(\mathcal E)$ of degree $2$ on ${\bf Perf}$. 
According to \cite[proof of Theorem 2.13]{PTVV} its underlying $2$-form (of degree $2$) can be described fairly easily (up to scaling). 
Recall that $\mathbb{T}_{\bf Perf}=\mathcal A[1]$, where $\mathcal A:=\mathcal Hom(\mathcal E,\mathcal E)$. Then the underlying 
(degree $2$) $2$-form of $Ch(\mathcal E)_2$ is given by 
$$
S^2_{\mathcal O_{\bf Perf}}\big(\mathbb{T}_{\bf Perf}[-1]\big)=
\xymatrix@C=5em{
S^2_{\mathcal O_{\bf Perf}}(\mathcal A) \ar[r] & \mathcal O_{\bf Perf} \\
\otimes^2_{\mathcal O_{\bf Perf}}(\mathcal A) \ar[u]\ar[r]^{~~~\mathrm{product}} & \mathcal A\ar[u]_{\mathrm{trace}}
}
$$
which is obviously non-degenerate. Therefore $Ch(\mathcal E)_2\in\mathbb{S}\mathrm{ymp}({\bf Perf},2)_0$. 

\subsubsection{A $1$-symplectic structure on $[\mathfrak{g}^*/G]$}\label{sss-1.2.3}

Let $G$ be a group scheme, with Lie algebra $\mathfrak{g}$. 
We consider the quotient stack $\X=[\mathfrak{g}^*/G]$, where $G$ acts on $\mathfrak{g}^*$ by the coadjoint action. 
Then $\mathbb{L}_\X=\mathcal O_{\mathfrak g^*}\otimes(\mathfrak{g}\oplus\mathfrak g^*[-1])$ as a $G$-equivariant 
quasi-coherent complex on $\mathfrak{g}^*$, with differential $d_{\mathfrak g}$ being the adjoint action $\mathfrak g\ni x\mapsto ad_x$ 
($ad_x$ is viewed as a linear function on $\mathfrak g^*$ with values in $\mathfrak g^*$). 

Observe that the canoncial element in $\mathfrak g^*\otimes \mathfrak g$ defines an element in 
$\big(S^2(\mathfrak{g}[1]\oplus\mathfrak g^*)[-1]\big)^G$, and thus a $G$-equivariant map 
$\mathcal O_{\mathfrak g^*}\to \mathcal O_{\mathfrak g^*}\otimes S^2(\mathfrak{g}[1]\oplus\mathfrak g^*)[-1]
=S^2_{\mathcal O_{\mathfrak g^*}}(\mathbb{L}_\X[1])[-1]$ which itself determines a point $\omega$ in $\mathcal A^2(\X,1)$. 
It is obviously non-degenerate (because the canonical element is) and it moreover canonically lifts to a point in 
$\mathcal A^{2,cl}(\X,1)$, providing a $1$-symplectic structure on $\X$. 
\begin{Rem}
Observe that $T^*[1](BG)\cong[\mathfrak{g}^*/G]$. What we have just described is nothing but the usual $1$-symplectic structure 
on a shifted cotangent stack (see \cite[Proposition 1.21]{PTVV}). 
\end{Rem}

\subsubsection{Shifted symplectic structures on mapping stacks}

Let $\Sigma$ be an {\it $\mathcal O$-compact} derived stack, following \cite[Defintion 2.1]{PTVV}. 
It ensures the existence of a natural $B\mathbb{G}_a$-equivariant morphism
$$
\Gamma\big(\Sigma\times\X,\Omega_{\Sigma\times\X}\big)\longrightarrow\Gamma(\Sigma,\mathcal O_\Sigma)\otimes\Gamma\big(\X,\Omega_{\X}\big)
$$ 
within $\mathbb{G}_m\textrm{-}Cpx$, where $\mathcal O_\Sigma$ has weight zero and is acted on trivially by $B\mathbb{G}_a$. 

If we further assume that we are given a ``fundamental class'' $[\Sigma]:\Gamma(\Sigma,\mathcal O_\Sigma)\to \kk[-d]$, then one gets 
a natural $B\mathbb{G}_a$-equivariant morphism
$$
\Gamma\big(\Sigma\times\X,\Omega_{\Sigma\times\X}\big)\longrightarrow\Gamma\big(\X,\Omega_\X\big)[-d]
$$
which in particular induces maps $\int_{[\Sigma]}:\mathcal A^{p(,cl)}(\Sigma\times\X,n)\to \mathcal A^{p(,cl)}(\X,n-d)$. 
We say that $[\Sigma]$ is a {\it $d$-orientation} if for any perfect complex $E$ the pairing
$$
\Gamma(\Sigma,E)\otimes\Gamma(\Sigma,E^\vee)\overset{\mathrm{duality}}{\longrightarrow}\Gamma(\Sigma,\mathcal O_\Sigma)
\overset{[\Sigma]}{\longrightarrow}\kk[-d]
$$
is non-degenerate. 
\begin{Thm}[\cite{PTVV},Theorem 2.6]\label{thm-mapping}
For any derived Artin stack $\Y$ and any $d$-oriented $\mathcal O$-compact derived stack $(\Sigma,[\Sigma])$ such that 
${\bf Map}(\Sigma,\Y)$ is itself a derived Artin stack, there is a map 
$$
\int_{[\Sigma]}ev_\Sigma^*(-):\mathbb{S}\mathrm{ymp}(\Y,n)\longrightarrow\mathbb{S}\mathrm{ymp}\big({\bf Map}(\Sigma,\Y),n-d\big)\,,
$$ 
where $ev_\Sigma:\Sigma\times{\bf Map}(\Sigma,\Y)\to \Y$ is the evaluation map. 
\end{Thm}
The procedure that appears in the above result is often refered to as {\it transgression}; one actually has transgression maps 
$\int_{[\Sigma]}ev_\Sigma^*(-):\mathcal A^{p(,cl)}(\Y,n)\to\mathcal A^{p(,cl)}\big({\bf Map}(\Sigma,\Y),n-d\big)$ for all $p$. 
Theorem \ref{thm-mapping} says that the transgression of a non-degenerate $2$-form is non-degenerate. 

\subsubsection{$1$-symplectic structures on $[G/G^{ad}]$}\label{sss-1.2.5}

Let $G$ be a reductive group scheme and let $\X=[G/G^{ad}]$, where $G^{ad}$ denotes the action of $G$ on itslef by conjugation. 
Then $\mathbb{L}_\X=T^*_G\oplus \mathcal O_G\otimes\mathfrak g^*[-1]$ as a $G^{ad}$-equivariant quasi-coherent complex on $G$, 
with differential $d_{LR}$ being adjoint to the infinitesimal action $\mathfrak g\to T_G$ which sends $x\in\mathfrak g$ to 
the sum $v_x:=x^L+x^R$ of the corresponding left and right invariant vector fields. 
Thus $\Omega_\X$ consists of $G$-equivariant functions on $\mathfrak g$ with values in $\Omega_G$, with differential $d_{LR}$ being given by 
$$
d_{LR}(f)(x)=-\iota_{v_x}\big(f(x)\big)\qquad(f\in\Omega_X\,,x\in\mathfrak g)\,.
$$
%Hencefore,  
%$$
%\mathcal A^p(\X,n)=\mathbb{M}\mathrm{ap}\left(\kk\,,
%\bigoplus_{q+r=p}\Big(\Gamma(G,\Omega^q_G)\otimes S^r(\mathfrak g^*)\Big)^G[n-r]\right)\,.
%$$
We now describe an interesting family of $2$-forms of degree $1$ on $\X$. Let $\theta=g^{-1}dg$ and $\bar\theta=dgg^{-1}$ be the left 
and right Maurer-Cartan forms on $G$. We then define $\beta:=\frac12(\theta+\bar\theta)\in\big(\Gamma(G,\Omega^1_G)\otimes\mathfrak g\big)^G$. 
If we are given a pairing $\langle\,,\rangle\in S^2(\mathfrak g^*)^G$ then we get a $G$-equivariant map $\omega_0(x):=\langle\beta,x\rangle$. 
We claim that $\omega_0\in\mathcal A^2(\X,1)_0$: for any $x\in\mathfrak{g}$, \\[-0.5cm]
$$
d_{LR}(\omega_0)(x)=-\langle\beta(x^L+x^R),x\rangle=0\,.
$$
One can actually check that if $\langle\,,\rangle$ is non-degenerate then so is $\omega_0$. 
\begin{Rem}
Recall that $[G/G^{ad}]={\bf Map}\big((S^1)_B,BG\big)$. The (degree $1$) $2$-form we've just described is actually the one which is 
obtained by transgressing the (degree $2$) $2$-form on $BG$ that is determined by $\langle\,,\rangle$. The latter being closed (one should 
actually rather write closable), so is the former. Below we give the ``key'' that closes it. 
\end{Rem}
\begin{Claim}
$\omega_1:=\frac{1}{12}\langle\theta,[\theta,\theta]\rangle=\frac{1}{12}\langle\bar\theta,[\bar\theta,\bar\theta]\rangle$ closes 
the $2$-form $\omega_0$. 
\end{Claim}
\begin{proof}
We first have to explain what this statement means. One can observe that the complex of closed $p$-forms is 
$\Omega_\X^{p,cl}=\oplus_{i\geq0}\Omega_\X^{p+2i}u^i$ with differential being $d_{LR}+u\cdot d$, where $d$ is the 
usual de Rham differential on $\Omega_G$. 

Notice that $\omega_1$ is a $3$-form of degree $0$, so that $\omega:=\omega_0+u\omega_1$ is homogeneous for the homological degree and 
for the weight. Hence we should prove that $\omega$ is closed, which is a standard calculation in the Cartan model for equivariant cohomology. 
\end{proof}

%%%%%%%%%%%%%%%%%%%%%%%%%%%%%

\section{Lagrangian structures}

\subsection{Recollection}

In this Subsection we again recollect some definitions and results from \cite{PTVV}. 

\subsubsection{Isotropic and Lagrangian structures}

Let $(\X,\omega)$ be a derived Artin stack equipped with an $n$-symplectic structure, 
and let ${\bf L}\overset{f}{\longrightarrow}\X$ be a morphism of derived stacks. The space 
$\mathbb{I}\mathrm{sot}(f,\omega)$ of {\it isotropic structures} on $f$ is the space 
$\mathcal A^{2,cl}({\bf L},n)_{f^*\omega,0}$ of paths from $f^*\omega$ to $0$ in $\mathcal A^{2,cl}({\bf L},n)$. 

\medskip

Let us then define the symplectic orthogonal sheaf 
$\mathbb{T}_{\bf L}^{f,\omega}:=(f^*\mathbb{T}_\X)\overset{h}{\underset{\mathbb{L}_{\bf L}[n]}\times}0$ 
to be the homotopy fiber of the morphism $f^*\mathbb{T}_\X\to \mathbb{L}_{\bf L}[n]$ given by the underlying $2$-form of $\omega$. 
Notice that any path $\gamma$ from $f^*\omega$ to $0$ in $\mathcal A^{2}({\bf L},n)$ induces a morphism 
$\mathbb{T}_{\bf L}\longrightarrow\mathbb{T}_{\bf L}^{f,\omega}$ that makes the following diagram homotopy commutative: 
$$
\vcenter{\xymatrix{
\mathbb{T}_{\bf L} \ar[dr]\ar[r] & \mathbb{T}_{\bf L}^{f,\omega} \ar[d]\\
& f^*\mathbb{T}_\X}}
$$
We say that $\gamma$ is {\it non-degenerate} if this morphism is an isomorphism in $hQCoh({\bf L})$. 

A {\it Lagrangian} structure on $f$ is an isotropic structure of which the underlying path in $\mathcal A^{2}({\bf L},n)$ is non-degenerate. 
We write $\mathbb{L}\mathrm{agr}(f,\omega)$ for the space of Lagrangian structures on $f$, defined as the (homotopy) pull-back
$$
\mathbb{L}\mathrm{agr}(f,\omega)
=A^2(\X,n)_{f^*\omega,0}^{nd}\!\!\underset{A^2(\X,n)_{f^*\omega,0}}{\times}\!\!\mathbb{I}\mathrm{sot}(f,\omega)
=A^2(\X,n)_{f^*\omega,0}^{nd}\!\!\overset{h}{\underset{A^2(\X,n)_{f^*\omega,0}}{\times}}\!\!\mathbb{I}\mathrm{sot}(f,\omega)\,.
$$
\begin{Rem}[Comparison with Lagrangian structures in \cite{PTVV}]\label{rem-2.1}
The definition of a Lagrangian structure that is given in \cite[Definition 2.8]{PTVV} might look different from ous, but it is equivalent. 
Namely, an isotropic structure provides a homotopy commutativity data for the following square:  
\begin{equation}\label{eq-isotropic}
\vcenter{\xymatrix{
\mathbb{T}_{\bf L} \ar[r]\ar[d] & 0 \ar[d]\\
f^*\mathbb{T}_\X\ar[r] & \mathbb{L}_{\bf L}[n]}}
\end{equation}
Being Lagrangian is equivalent to homotopy Cartesianity of the above square. This turns out to be equivalent to 
the induced morphism $\mathbb{T}_f\longrightarrow \mathbb{L}_{\bf L}[n-1]$ being an isomorphism in $hQCoh({\bf L})$, where $\mathbb{T}_f$ is the 
relative tangent complex. This last condition is the one which is given in \cite[Definition 2.8]{PTVV}. 
\end{Rem}
\begin{Rem}
Observe that there is yet another formulation of non-degeneracy. An isotropic structure provides a homotopy commutativity data 
for the following triangle: 
\begin{equation}\label{eq-isotropic2}
\vcenter{\xymatrix{
\mathbb{T}_{\bf L} \ar[r]\ar[rd]_0 & f^*\mathbb{L}_X[n] \ar[d]\\
 & \mathbb{L}_{\bf L}[n]}}
\end{equation}
Hence the horizontal map in \eqref{eq-isotropic2} lifts to a map $\mathbb{T}_{\bf L}\longrightarrow\mathbb{L}_f[n-1]$. 
The isotropic structure is Lagrangian (i.e.~non-degenerate) if and only if this map is an isomorphism in $hQCoh({\bf L})$. 
\end{Rem}
\begin{Exa}
It is worth noticing that a Lagrangian structure on $\X\to\bullet_{(n)}$, where $\bullet_{(n)}$ denotes the point $\bullet$ with 
its canonical (and unique) $n$-symplectic structure, is nothing else than an $(n-1)$-symplectic structure on $\X$. 
\end{Exa}

\subsubsection{Shifted symplectic structures on derived Lagrangian intersections}

Let ${\bf L}_1\overset{f_1}{\longrightarrow}\X$ and ${\bf L}_2\overset{f_2}{\longrightarrow}\X$ be morphisms of derived Artin stacks, 
and assume that $\X$ is equipped with an $n$-symplectic structure $\omega$. 
\begin{Thm}[\cite{PTVV},Theorem 2.10]\label{thm-derint}
There is a map 
$$
\mathbb{L}\mathrm{agr}(f_1,\omega)\times\mathbb{L}\mathrm{agr}(f_2,\omega)
\longrightarrow\mathbb{S}\mathrm{ymp}\left({\bf L}_1\overset{h}{\underset{\X}{\times}}{\bf L}_2,n\right)\,.
$$
\end{Thm}
We now give an explanation for the above result, which we believe is quite enlighting. 
First of all one observes that it is sufficient to prove Theorem \ref{thm-derint} for self-intersections (namely, one considers 
a component in the self-intersection of ${\bf L}:={\bf L}_1\coprod {\bf L}_2$). 
Then notice that for any morphism ${\bf L}\to\X$ one has a fiber sequence 
$$
\mathcal A^{p,cl}(\X/{\bf L},n)\longrightarrow \mathcal A^{p,cl}(\X,n)\longrightarrow \mathcal A^{p,cl}({\bf L},n)\,,
$$
where $\X/{\bf L}=\X\overset{h}{\underset{{\bf L}}{\coprod}}\bullet$. One can show that given an $n$-symplectic structure $\omega$ on $\X$, 
a Lagrangian structure on ${\bf L}\to\X$ provides a lift of $\omega$ to an $n$-symplectic structure $\widetilde\omega$ on $\X/{\bf L}$. 

Then recall that ${\bf L}\overset{h}{\underset{\X}{\times}}{\bf L}$ can be interpreted as the stack of paths in $\X$ with both ends in 
${\bf L}$, which may be identified with the stack of pointed loops in $\X/{\bf L}$. The stack of pointed loops in $\X/{\bf L}$ is 
finally $(n-1)$-symplectic by a variant Theorem \ref{thm-mapping}. 

\subsection{Examples of Lagrangian structures}

The main goal of this subsection is to provide new examples of Lagrangian structures. 

\subsubsection{Lagrangian morphisms to $[\mathfrak{g}^*/G]$}

Let $X$ be an actual smooth scheme equipped with a (left) $G$-action and a $G$-equivariant morphism $\mu:X\to\mathfrak{g}^*$, 
where $G$ is a group scheme with Lie algebra $\mathfrak g$. This determines a morphism $[X/G]\overset{f}{\longrightarrow}[\mathfrak{g}^*/G]$. 
We describe in this \S~interesting Lagrangian structures on $f$ having some genuine symplectic geometric origin. 
The $1$-symplectic structure we consider on $[\mathfrak{g}^*/G]=T^*[1](BG)$ is the one described in \S\ref{sss-1.2.3}. 
Recall that $\mathbb{L}_{[X/G]}=T^*_X\oplus \mathcal O_X\otimes\mathfrak g^*[-1]$ 
as a $G$-equivariant quasi-coherent complex on $X$, with differential $\vec{d}$ being adjoint to the infinitesimal action 
$\mathfrak g\to T_X$ which we denote by $x\mapsto\vec{x}$. 
Thus $\Omega_{[X/G]}$ consists of $G$-equivariant functions on $\mathfrak g$ with values in $\Omega_X$, with differential $\vec{d}$ 
being given by 
$$
\vec{d}(h)(x)=-\iota_{\vec{x}}\big(f(x)\big)\qquad(h\in\Omega_{[X/G]}\,,x\in\mathfrak g)\,.
$$
Borrowing the notation from \S\ref{sss-1.2.3}, a path from $f^*\omega$ to $0$ in $\mathcal A^{2}([X/G],1)$ is then the same as 
the data of a $2$-form $\gamma$ of degree $0$ on $[X/G]$ such that $\vec{d}(\gamma)=-f^*\omega$; that is to say 
a $G$-invariant $2$-form $\gamma$ on $X$ such that $\iota_{\vec{x}}\gamma=\mu^*dx$, where we view $x\in\mathfrak g$ as a linear function 
on $\mathfrak g^*$. The non-degeneracy condition boils down to the usual non-degeneracy of the $2$-form $\gamma$. Moreover, if 
$\gamma$ is closed for the usual de Rham differential on $X$ then our path admits a canonical lift to $\mathcal A^{2,cl}([X/G],1)$. 

\medskip

Observe that what we have just described (a genuine non-degenerate closed $2$-form $\gamma$ on $X$ such that $\iota_{\vec{x}}\gamma=\mu^*dx$) 
is nothing else than a {\it moment map} (or, {\it Hamiltonian}) {\it structure} on $\mu$. 

\subsubsection{Lagrangian morphisms to $[G/G^{ad}]$}\label{sss-2.2.2}

Let $X$ be an actual smooth scheme equipped with a (left) $G$-action and a $G$-equivariant morphism $\mu:X\to G$, where $G$ is a 
reductive group scheme. This determines a morphism $[X/G]\overset{f}{\longrightarrow}[G/G^{ad}]$. 
As above we now describe some interesting Lagrangian structures on $f$, where the symplectic structure on $[G/G^{ad}]$ is the one 
determined by a given invariant symmetric bilinear form $\langle\,,\rangle$ on $\mathfrak g$ (see \S\ref{sss-1.2.5}). 

Borrowing the notation from the previous \S~as well as from \S\ref{sss-1.2.5}, a path from $f^*\omega_0$ to $0$ in 
$\mathcal A^{2}([X/G],1)$ is then the same as the data of a $2$-form $\gamma$ of degree $0$ on $[X/G]$ such that 
$\vec{d}(\gamma)=-f^*\omega_0$; that is to say a $G$-invariant $2$-form $\gamma$ on $X$ such that 
$\iota_{\vec{x}}\gamma=\mu^*\langle\beta,x\rangle$. 
\begin{Claim}
The path is non-degenerate if and only if $\ker(\gamma)\cap\ker(d\mu)=0$. 
\end{Claim}
\begin{proof}
First observe that $\mathbb{T}_f\cong\mathbb{T}_\mu$ is the two-step $G$-equivariant complex $T_X\oplus \mu^*T_G[-1]$ on $X$ with 
differential being $d\mu$. Then recall that $\mathbb{L}_{[X/G]}=T^*_X\oplus\mathcal O_X\otimes\mathfrak g^*[-1]$ with differential $\vec{d}$. 

The cochain map $\mathbb{T}_f\to\mathbb{L}_{[X/G]}$ induced by the path (see Remark \ref{rem-2.1}) has the following form: \\[-0.5cm]
\begin{itemize}
\item in degree $0$ it is simply the $G$-equivariant map $T_X\to T^*_X$ induced by the $G$-invariant $2$-form $\gamma$, \\[-0.6cm]
\item in degree $1$ it is an isomorphism (essentially given by $\langle\beta,-\rangle$). \\[-0.5cm]
\end{itemize}
Being an isomorphism in degree $1$, it is a quasi-isomorphism if and only if its restriction on degree $0$ cocycles is injective. 
This is precisely the condition that $\ker(\gamma)\cap\ker(d\mu)=0$. 
\end{proof}
Let us finally guess what one could require on $\gamma$ in order to lift it to a path in $A^{2,cl}\big([X/G],1\big)$ from $f^*\omega$ to $0$:
a sufficient requirement is that $d_X\gamma=-\mu^*\omega_1$, where $d$ is the de Rham differential on $X$. 

\medskip

Observe that the data of a genuine invariant $2$-form $\gamma$ on $X$ such that $\iota_{\vec{x}}\gamma=\mu^*\langle\beta,x\rangle$ for any 
$x\in\mathfrak g$, $d\gamma=-\mu^*\omega_1$, and satisfying the above non-degeneracy condition is precisely what is called a 
{\it Lie group valued moment map} (or, {\it quasi-Hamiltonian}) {\it structure} on $\mu$ (see \cite[Definition 2.2]{AMM}). 

\subsubsection{Lagrangian structures on mapping stacks}

Let $f:\Upsilon\to\Sigma$ be a morphism between $\mathcal O$-compact derived stacks, and assume that $\Upsilon$ is equipped 
with a fundamental class $[\Upsilon]:\Gamma(\Upsilon,\mathcal O_\Upsilon)\to \kk[-d]$. 
\begin{Def}
The space $\mathbb{B}{\rm nd}(f,[\Upsilon])$ of {\it boundary structures} on $f$ is the space of paths from $f_*[\Upsilon]$ 
to $0$ in $\map\big(\Gamma(\Sigma,\mathcal O_\Sigma),\kk[-d]\big)$. 
\end{Def}
Hence any boundary structure on $f$ gives rise to a homotopy between the induced map 
$$
\int_{f_*[\Upsilon]}:\mathcal A^{p(,cl)}(\Sigma\times\X,n)\to \mathcal A^{p(,cl)}(\X,n-d)
$$
and the zero map for any derived Artin stack $\X$. 

\medskip

Let $\Y$ be a derived Artin stack with a (closed) $p$-form $\omega$ of degree $n$. 
Then we have a (closed) $p$-form $\int_{[\Upsilon]}ev_\Upsilon^*\omega$ of degree $n-d$ on ${\bf Map}(\Upsilon,\Y)$. 
\begin{Claim}\label{exa-2.5}
There is a path in $\mathcal A^{p(,cl)}\big({\bf Map}(\Sigma,\Y),n-d\big)$ between its pull-back along the 
morphism $rest:{\bf Map}(\Sigma,\Y)\to{\bf Map}(\Upsilon,\Y)$ and $0$. 
\end{Claim}
\begin{proof}
The claim is a consequence of 
$$
rest^*\int_{[\Upsilon]}ev_\Upsilon^*\omega=\int_{[\Upsilon]}({\rm id}\times rest)^*ev_\Upsilon^*\omega=\int_{f_*[\Upsilon]}ev_\Sigma^*\omega\,,
$$
where the second equality follows from the commutativity of \\[0.1cm]
$
\vcenter{\xymatrix@C=5em{
~~~&~~~&\Upsilon\times{\bf Map}(\Sigma,\Y) \ar[d]^{f\times{\rm id}}\ar[r]^{{\rm id}\times rest} & 
\Upsilon\times{\bf Map}(\Upsilon,\Y) \ar[d]^{ev_\Upsilon}  \\
~~~&~~~&\Sigma\times{\bf Map}(\Sigma,\Y) \ar[r]^{ev_\Sigma} & \Y\,. }
}$
\end{proof}

\medskip

Given a boundary structure on $f:\Upsilon\to\Sigma$ and a perfect $E$ on $\Sigma$, we get a pairing 
$$
\Gamma(\Sigma,E)\otimes\Gamma(\Upsilon,f^*E^\vee)\longrightarrow \Gamma(\Upsilon,f^*E)\otimes\Gamma(\Upsilon,f^*E^\vee)
\overset{\mathrm{duality}}{\longrightarrow}\Gamma(\Upsilon,\mathcal O_\Upsilon)\overset{[\Upsilon]}{\longrightarrow}\kk[-d]\,.
$$
We write $\Gamma(\Sigma,E)^\perp$ for the homotopy fiber of the induced map $\Gamma(\Upsilon,f^*E^\vee)\to \Gamma(\Sigma,E)^\vee[-d]$. 
The path $\gamma$ between $f_*[\Upsilon]$ and $0$ thus induces a map $\Gamma(\Sigma,E^\vee)\to\Gamma(\Sigma,E)^\perp$ which makes the 
following diagram homotopy commutative: 
$$
\vcenter{\xymatrix{
\Gamma(\Sigma,E^\vee) \ar[dr]\ar[r] & \Gamma(\Sigma,E)^\perp \ar[d]\\
& \Gamma(\Upsilon,f^*E^\vee)}}
$$
\begin{Def}
We say that a boundary structure $\gamma$ on $f$ is {\it non-degenerate} if the induced map $\Gamma(\Sigma,E^\vee)\to\Gamma(\Sigma,E)^\perp$ 
is an isomorphism in $hCpx$ for any perfect $E$. 
A {\it relative $d$-orientation} on $f:\Upsilon\to\Sigma$ is a $d$-orientation of $\Upsilon$ together with a non-degenerate boundary structure. 
\end{Def}
We then have the following ``relative'' analog of Theorem \ref{thm-mapping}: 
\begin{Thm}\label{thm-maplag}
Let $\Y$ be a derived Artin stack equipped with an $n$-symplectic structure, $f:\Upsilon\to\Sigma$ be a morphism of $\mathcal O$-compact 
derived stacks and $[\Upsilon]$ be a $d$-orientation on $\Sigma$. If the mapping stacks $\Map(\Sigma,\Y)$ and $\Map(\Upsilon,\Y)$ are derived 
Artin stacks then we have a map 
$$
\mathbb{B}{\rm nd}\left(f,[\Upsilon]\right)\longrightarrow \mathbb{I}{\rm sot}\left(rest,\int_{[\Upsilon]}ev_\Upsilon^*\omega\right)
$$
which sends non-degenerate boundary structures to Lagrangian structures. 
\end{Thm}
Here $rest:{\bf L}=\Map(\Sigma,\Y)\to\Map(\Upsilon,\Y)=\X$ is again given by composing with $f$. 
\begin{proof}
We have already seen in Claim \ref{exa-2.5} the existence, for any (closed) $p$-form $\omega$ of degree $n$ on $\Y$, of a map 
$\mathbb{B}{\rm nd}\left(f,[\Upsilon]\right)$ to paths between $rest^*\int_{[\Upsilon]}ev_\Upsilon^*\omega$ and $0$ in 
$\mathcal A^{p(,cl)}({\bf L},n-d)$. 

It remains to show that, whenever $\omega$ is a $2$-form, the above map sends non-degenerate boundary structures to non-degenerate isotropic 
structures. Note that we already know from Theorem \ref{thm-mapping} that $\tilde\omega=\int_{[\Upsilon]}ev_\Upsilon^*\omega$ is non-degenerate. 

For this purpose we describe the map $\mathbb{T}_{{\bf L}}\to \mathbb{T}_{{\bf L}}^{rest,\tilde\omega}$ at a given $A$-point 
$x:Spec(A)\to{\bf L}$  (i.e.~a map $\xi:\Sigma_A=\Sigma\times Spec(A)\to\Y$). First observe that 
$\mathbb{T}_{{\bf L},x}=\Gamma(\Sigma_A,\xi^*\mathbb{T}_\Y)$ and that $\mathbb{T}_{{\bf L},x}^{rest,\tilde\omega}$ is the homotopy fiber of 
the map $rest^*\mathbb{T}_{\X,x}\longrightarrow\mathbb{L}_{{\bf L},x}[n-d]$, which can be decomposed as 
$$
rest^*\mathbb{T}_{\X,x}=\Gamma\big(\Upsilon_A,(f\times{\rm id})^*\xi^*\mathbb{T}_\Y\big)\tilde\longrightarrow
\Gamma\big(\Upsilon_A,(f\times{\rm id})^*\xi^*\mathbb{L}_\Y\big)[n]
\longrightarrow\Gamma(\Sigma_A,\xi^*\mathbb{L}_\Y^\vee)^\vee[n-d]=\mathbb{L}_{{\bf L},x}[n-d]\,,
$$
where the first arrow is induced by $\omega$ and the second one is induced by $[\Upsilon]$. Hence $\mathbb{T}_{{\bf L},x}^{rest,\tilde\omega}$ 
can be identified, using $\omega$, with $\Gamma(\Sigma_A,\xi^*\mathbb{L}_\Y^\vee)^\perp[n]$ and the map 
$\mathbb{T}_{{\bf L},x}\longrightarrow\mathbb{T}_{{\bf L},x}^{rest,\tilde\omega}$ decomposes as 
$$
\Gamma(\Upsilon_A,\xi^*\mathbb{T}_\Y)\cong\Gamma(\Upsilon_A,\xi^*\mathbb{L}_\Y)[n]\longrightarrow
\Gamma(\Sigma_A,\xi^*\mathbb{L}_\Y^\vee)^\perp[n]\cong\mathbb{T}_{{\bf L},x}^{rest,\tilde\omega}\,.
$$
The boundary structure being non-degenerate, it is an isomorphism in $h(A\textrm{-mod})$. 
\end{proof}

\subsection{Symplectic structures on mapping stacks with boundary conditions}

\subsubsection{Mapping stacks with Lagrangian target}

Let $f:{\bf L}\to{\bf Y}$ be a a morphism of derived Artin stack and let $(\Sigma,[\Sigma])$ be a $d$-oriented $\mathcal O$-compact 
derived Artin stack such that $\Map(\Sigma,{\bf L})$ and $\Map(\Sigma,{\bf Y})$ are themselves derived Artin stacks. 
We have the following obvious extension of Theorem \ref{thm-mapping} (from which we borrow the notation): 
\begin{Thm}\label{thm-maplag2}
If $\omega$ is an $n$-symplectic structure on ${\bf Y}$, then $\int_{[\Sigma]}ev^*(-)$ defines a map
$$
\mathbb{L}{\rm agr}(f,\omega)\longrightarrow\mathbb{L}{\rm agr}\left(f\circ-,\int_{[\Sigma]}ev^*(\omega)\right)\,.
$$
\end{Thm}

\subsubsection{Mapping stacks with boundary conditions}

Let $f:\Upsilon\to\Sigma$ and $g:{\bf L}\to{\bf Y}$ be a morphisms of derived Artin stacks. We consider the {\it relative derived mapping stack} 
$$
\Map(f,g):=\Map(\Upsilon,{\bf L})\!\overset{h}{\underset{\Map(\Upsilon,{\bf Y})}{\times}}\!\Map(\Sigma,{\bf Y})
$$
and assume that it is a derived Artin stack. 
\begin{Thm}\label{thm-2.13}
If $f$ carries a relative $d$-orientation, ${\bf Y}$ carries an $n$-symplectic structure and $g$ carries a Lagrangian structure, then 
$\Map(f,g)$ has an $(n-d-1)$-symplectic structure. 
\end{Thm}
\begin{proof}
First of all, the $d$ orientation on $\Upsilon$ and the $n$-symplectic structure on ${\bf Y}$ produce an $(n-d)$-symplectic structure on 
$\Map(\Upsilon,{\bf Y})$ (by Theorem \ref{thm-mapping}). Then the relative $d$-orientation on $f$ produces a Lagrangian structure on 
$\Map(\Sigma,{\bf Y})$ (by Theorem \ref{thm-maplag}) and the Lagrangian structure on $g$ produces a Lagrangian structure on 
$\Map(\Upsilon,{\bf L})$ (by Theorem \ref{thm-maplag2}). Hence, by Theorem \ref{thm-derint} on derived Lagrangian intersections, 
$\Map(f,g)$ inherits an $(n-d-1)$-symplectic structure. 
\end{proof}

%%%%%%%%%%%%%%%%%%%%%%%%%%%%%

\section{Recovering usual symplectic and Lagrangian moduli stacks}

\subsection{Topological context}

\subsubsection{Boundary structures of Betti type}

Recall that we have a symmetric monoidal $(\infty,1)$-functor $(-)_B:Top\to dSt_{{\bf k}}$. 
For a space $\mathbb{X}$ we have that $\Gamma(\mathbb{X}_B,\mathcal O_{\mathbb{X}_B})=C^*_{sing}(\mathbb{X},\kk)$.
More generally $QCoh(\mathbb{X}_B)$ is equivalent to the $(\infty,1)$-category of locally constant sheaves of $\kk$-modules on $\mathbb{X}$, 
and $\Gamma(\mathbb{X}_B,E)=\Gamma(\mathbb{X},E)$ for any object $E$. 
Note that if $\mathbb{X}$ is compact then $\mathbb{X}_B$ is $\mathcal O$-compact. 

We now let $M$ be a compact oriented topological manifold of dimension $d+1$ with boundary $\partial M$, and consider 
the morphism of derived Artin stacks $f:\Upsilon:=(\partial M)_B\longrightarrow M_B=:\Sigma$ induced by the inclusion of the boundary. 

First of all observe that $\partial M$ being closed and oriented, it comes equipped with a fundamental class 
$[\partial M]\in H_{d}(\partial M,\kk)$ and hence gives a fundamental class $[\Upsilon]:\Gamma(\Upsilon,\mathcal O_\Upsilon)\to \kk[-d]$ 
for $\Upsilon$. Poincar\'e duality guaranties that $[\Upsilon]$ is a $d$-orientation. 

Then the orientation on $M$ provides a relative fundamental class $[M]\in H_{d+1}(M,\partial M,\kk)$ that is sent to $[\partial M]$ 
{\it via} the boundary map $H_{d+1}(M,\partial M,\kk)\to H_{d}(\partial M,\kk)$. Hence\footnote{Because $C_*(M,\partial M,\kk)[-1]$ is the 
homotopy fiber of the map $C_*(\partial M,\kk)\to C_*(M,\kk)$. } it determines a path from $f_*[\Upsilon]$ to $0$ 
in $\map\big(\Gamma(\Sigma,\mathcal O_\Sigma),\kk[-d]\big)$. Relative Poincar\'e duality guaranties that the boundary 
structure we have just defined on $f:\Upsilon\to \Sigma$ is non-degenerate. 

\subsubsection{Mapping stacks with Betti source}

We borrow the notation from the previous paragraph. 
Note that if $\mathbb{X}$ is compact then ${\bf Map}(\mathbb{X}_B,{\bf Y})$ is a derived Artin stack if ${\bf Y}$ is. 
If moreover ${\bf Y}$ carries an $n$-symplectic structure, then Theorem \ref{thm-mapping} tells us that a fundamental class 
$[\partial M]\in H_{d}(\partial M,\kk)$ determines an $(n-d)$-symplectic structure on ${\bf Map}(\Upsilon,{\bf Y})$. 
Additionally, Theorem \ref{thm-maplag} tells us that a fundamental class $[M]\in H_{d+1}(M,\partial M,\kk)$ determines a Lagrangian 
structure on the restriction morphism $rest:{\bf Map}(\Sigma,{\bf Y})\to{\bf Map}(\Upsilon,{\bf Y})$. 

\medskip

In particular, if $G$ is a reductive group and ${\bf Y}=BG$ we get a $(2-d)$-symplectic structure on 
${\bf Map}(\Upsilon,BG)=:{\bf Loc}_G(\partial M)$. Moreover, we have a Lagrangian structure on the restriction morphism 
${\bf Loc}_G(M)\to{\bf Loc}_G(\partial M)$. 
\begin{Exa}
When $d=2$, we get in particular a $0$-shifted symplectic structure on the derived moduli stack ${\bf Loc}_G(\partial M)$ of $G$-local 
systems on a compact oriented surface $\partial M$. Recall (see \cite[\S3.1]{PTVV}) that it induces a genuine symplectic structure on the 
coarse moduli space $Loc_G^s(\partial M)$ of simple $G$-local systems on $\partial M$. 
Moreover, the existence of a Lagrangian structure on the restriction map ${\bf Loc}_G(M)\to{\bf Loc}_G(\partial M)$ tells us in particular 
that the regular locus of the subspace in $Loc_G^s(\partial M)$ consisting of those $G$-local systems extending to $M$ is a Lagrangian 
subvariety. This known fact is the starting point of the construction of the Casson invariant. 
\end{Exa}
\begin{Exa}
If $d=1$ then $\partial M$ is a sum of circles $(S^1)^{\coprod n}$ and thus ${\bf Map}(\Upsilon,BG)=[G/G^{ad}]^{\times n}$ is equipped 
with the $1$-symplectic structure of \S\ref{sss-1.2.5} and we have a Lagrangian structure on the morphism 
${\bf Loc}_G(M)={\bf Map}(\Sigma,BG)\longrightarrow[G/G^{ad}]^{\times n}$. Let now $O_1,\dots,O_n$ be conjugacy classes in $G$; then the morphism 
$[O_1/G^{ad}]\times\cdots\times[O_n/G^{ad}]\longrightarrow[G/G^{ad}]^{\times n}$ comes equipped with a Lagrangian structure (see \S\ref{sss-2.2.2}). 
Hence the derived fiber product 
$$
{\bf Loc}_G(M;O_1,\dots,O_n):={\bf Loc}_G(M)\overset{h}{\underset{[G/G^{ad}]^{\times n}}{\times}}\Big([O_1/G^{ad}]\times\cdots\times[O_n/G^{ad}]\Big)
$$
inherits a $0$-symplectic structure. If $\partial M\neq\emptyset$ then one can show that ${\bf Loc}_G(M)$ is a smooth Deligne-Mumford stack. 
Moreover, for a generic collection $(O_1,\dots,O_n)$ of conjugacy classes, the derived fiber product ${\bf Loc}_G(M;O_1,\dots,O_n)$ is also 
a smooth Deligne-Mumford stack. We recover in this way the symplectic structure on the moduli space of local systems on a surface with 
prescribed holonomy along the boundary components (see \cite[Section 9]{AMM} and references therein). 
\end{Exa}

\subsection{Algebro-geometric context}

\subsubsection{Boundary structures of algebro-geometric type}

Let $\Sigma$ be a geometrically connected smooth proper algebraic variety of dimension $d+1$ together with a smooth $d$-Calabi-Yau divisor 
$\Upsilon$ having anticanonical class. We have a fundamental class 
$$
[\Upsilon]:\Gamma(\Upsilon,\mathcal O_\Upsilon)\longrightarrow H^d(\Upsilon,\mathcal O_\Upsilon)[-d]\tilde\longrightarrow 
H^d(\Upsilon,\mathcal K_\Upsilon)[-d]\cong\kk[-d]\,,
$$
where $\mathcal K_\Upsilon:=\wedge^d(L_\Upsilon)$ is the canonical sheaf. This is actually a $d$-orientation by Serre duality. 

Then using that $\mathcal K_\Sigma:=\wedge^{d+1}(L_\Sigma)$ is isomorphic to $\mathcal O_\Sigma(-\Upsilon)$ we get a relative 
fundamental class 
$$
[\Sigma]:\Gamma\big(\Sigma,\mathcal O_\Sigma(-\Upsilon)\big)\longrightarrow H^{d+1}\big(\Sigma,\mathcal O_\Sigma(-\Upsilon)\big)[-d-1]
\tilde\longrightarrow H^{d+1}(\Sigma,\mathcal K_\Sigma)[-d-1]\cong\kk[-d-1]\,.
$$
which lifts $[\Upsilon]$. Namely, if we denote by $f$ the inclusion of $\Upsilon$ into $\Sigma$ then there is a short exact sequence 
\begin{equation}\label{eq:ses}
0\longrightarrow \mathcal O_\Sigma(-\Upsilon)\longrightarrow \mathcal O_\Sigma\longrightarrow f_*\mathcal O_\Upsilon\longrightarrow 0
\end{equation}
such that the induced map $b:\Gamma(\Upsilon,\mathcal O_\Upsilon)\longrightarrow\Gamma\big(\Sigma,O_\Sigma(-\Upsilon)\big)[1]$ makes the 
following diagram commute: 
$$
\xymatrix{
\Gamma(\Upsilon,\mathcal O_\Upsilon) \ar[r]^b\ar[d]_{[\Upsilon]} & \Gamma\big(\Sigma,O_\Sigma(-\Upsilon)\big)[1] \ar[d]^{[\Sigma]} \\
H^d(\Upsilon,\mathcal K_\Upsilon)[-d]\cong \kk[-d] \ar@{=}[r] & \kk[-d]\cong H^{d+1}(\Sigma,\mathcal K_\Sigma)[-d]\,.
}
$$
Exactness of \eqref{eq:ses} tells us in particular that $\Gamma(\Sigma,\mathcal O_\Sigma)$ is the homotopy fiber of $b$. 
Therefore $[\Sigma]$ provides a homotopy between $f_*[\Upsilon]$ and $0$ in $\map\big(\Gamma(\Sigma,\mathcal O_\Sigma),\kk[-d]\big)$, 
that is to say a boundary structure on  $f$. 
\begin{Claim}
This boundary structure is non-degenerate. 
\end{Claim}
\begin{proof}
Let $E$ be a perfect complex on $\Sigma$. We shall prove that the map $\Gamma(\Sigma,E^\vee)\to \Gamma(\Sigma,E)^\perp$ 
is an isomorphism in $hCpx$. Recall that $\Gamma(\Sigma,E)^\perp$ is the homotopy fiber of 
$$
\Gamma(\Upsilon,f^*E^\vee)\cong\Gamma(\Upsilon,f^*E)^\vee[-d]\longrightarrow\Gamma(\Sigma,E)^\vee[-d]\,,
$$
which is nothing but $\Gamma(\Sigma,E\otimes\mathcal K_\Sigma)^\vee[-d-1]$ (again by exactness of \eqref{eq:ses}). 

One can then easily check that the map $\Gamma(\Sigma,E^\vee)\to \Gamma(\Sigma,E\otimes\mathcal K_\Sigma)^\vee[-d-1]$ provided 
by the boundary structure is adjoint to the following pairing: 
$$
\Gamma(\Sigma,E^\vee)\otimes\Gamma(\Sigma,E\otimes\mathcal K_\Sigma)\overset{\mathrm{duality}}{\longrightarrow}\Gamma(\Sigma,\mathcal K_\Sigma)
\overset{[\Sigma]}{\longrightarrow}\kk[-d-1]\,.
$$
Hence the map $\Gamma(\Sigma,E^\vee)\to \Gamma(\Sigma,E)^\perp$ is an isomorphism in $hCpx$ by Serre duality. 
\end{proof}
We refer to \cite{To-survey} for a more general statement dealing with the case of a non-smooth divisor. 

\subsubsection{Mapping stack with Fano and Calabi-Yau sources}

First observe that any geometrically connected smooth proper algebraic variety $\Sigma$ (over $\kk$), considered as a derived stack, 
is $\mathcal O$-compact. Then ${\bf Map}(\Sigma,{\bf Y})$ is a derived Artin stack if ${\bf Y}$ is. Let us further assume that 
${\bf Y}$ carries an $n$-symplectic form. 

Using Theorems \ref{thm-mapping} and \ref{thm-maplag}, and borrowing the notation from the previous paragraph, we obtain an $(n-d)$-symplectic 
structure on ${\bf Map}(\Upsilon,{\bf Y})$ together with a Lagrangian structure on the restriction morphism
$rest:{\bf Map}(\Sigma,{\bf Y})\to {\bf Map}(\Upsilon,{\bf Y})$. 

\medskip

In particular, if $G$ is a reductive group and ${\bf Y}=BG$  we get a $(2-d)$-symplectic structure on 
${\bf Map}(\Upsilon,{\bf Y})=:{\bf Bun}_G(\Upsilon)$ and a Lagrangian structure on the restriction morphism 
${\bf Bun}_G(\Sigma)\to{\bf Bun}_G(\Upsilon)$. Similarly, if ${\bf Y}={\bf Perf}$ then we get $(2-d)$-symplectic structure on 
${\bf Perf}(\Upsilon)$ and a Lagrangian structure on the restriction morphism ${\bf Perf}(\Sigma)\to{\bf Perf}(\Upsilon)$. 

\begin{Exa}
If $d=1$ then $\Sigma$ is a del Pezzo surface and $\Upsilon$ is an elliptic curve. 
Hence ${\bf Bun}_G(\Upsilon)$ is $1$-symplectic. Consider a semi-stable $G$-bundle $E$, which determines a point in ${\bf Bun}_G(\Upsilon)$. 
Formally around that point ${\bf Bun}_G(\Upsilon)$ is isomorphic to the neighbourhood of the unit in $H/H^{ad}$, for some reductive 
subgroup $H\subset G$. Hence $E$ provides us with a Lagrangian morphism $BH\longrightarrow{\bf Bun}_G(\Upsilon)$. 
Note that the restriction morphism ${\bf Bun}_G(\Sigma)\to{\bf Bun}_G(\Upsilon)$ also has a Lagrangian structure. Therefore 
$$
{\bf Bun}_G(\Sigma;\Upsilon,E):={\bf Bun}_G(\Sigma)\overset{h}{\underset{{\bf Bun}_G(\Upsilon)}{\times}}BH\,,
$$
which is nothing but the derived moduli stack of $(\Upsilon,E)$-framed $G$-bundles on $\Sigma$, is $0$-symplectic. 
When $BG$ is replaced by ${\bf Perf}$ the genuine symplectic structure that we get on the smooth locus should coincide with the ones 
defined in \cite{Bot,Sala}. 
\end{Exa}
\begin{Exa}
If $d=2$ then $\Sigma$ is a Fano $3$-fold and $\Upsilon$ is a $K3$ surface. In \cite{PTVV} it is proven that the $0$-symplectic 
structure on ${\bf Map}(\Upsilon,{\bf Perf})={\bf Perf}(\Upsilon)$ induces an actual symplectic structure on the coarse moduli space 
$Perf^s(\Upsilon)$ of simple perfect complexes on $\Upsilon$. The existence of a Lagrangian structure on 
${\bf Perf}(\Sigma)\to{\bf Perf}(\Upsilon)$ tells us that the closed subspace of $Perf^s(\Upsilon)$ consisting of restrictions of perfect 
complexes on $\Sigma$ is Lagrangian (compare with \cite[Proposition 2.1]{T}). 
\end{Exa}

%%%%%%%%%%%%%%%%%%%%%%%%%%%%%

\section{Application: topological field theories from mapping stacks}

Let $M$ be a closed oriented topological $3$-manifold together with an emmbedded closed oriented surface $S$ that separates it 
into two part $M_+$ and $M_-$: $\partial M_+=S=\partial \overline{M_-}$. 

Let $G$ be a reductive group and observe that we have a weak equivalence of derived Artin stacks
$$
{\bf Loc}_G(M)\,\tilde\longrightarrow\,{\bf Loc}_G(M_+)\overset{h}{\underset{{\bf Loc}_G(S)}{\times}}{\bf Loc}_G(M_-)\,.
$$
One can actually prove that it preserves the $(-1)$-symplectic structures on both sides, where the one on the 
r.h.s.~is coming from the derived fiber product of Lagrangian morphisms. 

\medskip

We view the above as an instance of a more general fact that we prove in this Section: 
\begin{quote}
{\it The functor ${\bf Map}\big((-)_B,{\bf Y}\big)$ defines a topological field theory with values in a suitable category of Lagrangian 
correspondences whenever ${\bf Y}$ admits an $n$-symplectic structure. }
\end{quote}

\subsection{Classical TFTs from mapping stacks}\label{ss-TFTmap}

\subsubsection{The cobordism category}

Let $d\geq0$ be an integer. 
We define $dCob$, resp.~$dCob^{or}$, to be the category with objects being closed differentiable manifolds, resp.~oriented differentiable manifolds, 
and morphisms being diffeomorphism classes of cobordisms, resp.~oriented cobordisms. It is a symmetric monoidal category, with 
monoidal product the disjoint union. 

\medskip

A {\it $d$-dimensional TFT} with values in a symmetric monoidal category $\mathcal C$ is a symmetric monoidal functor $dCob\to \mathcal C$. 

\subsubsection{A category of cospans}\label{sss-cospans}

We let $Cosp$ be the category with object being compact spaces and ${\rm Hom}_{Cosp}(X,Y)$ being weak equivalence classes 
of cospans $X\rightarrow F\leftarrow Y$. 
Composition of morphisms is given by the homotopy push-out: for $X\rightarrow F\leftarrow Y$ and $Y\rightarrow G\leftarrow Z$, we define 
$X\rightarrow G\circ F\leftarrow Z$ to be $X\rightarrow F\overset{h}{\underset{Y}{\coprod}}G\leftarrow Z$. It is symmetric monoidal, 
with monoidal product the disjoint union (the categorical sum). 

\medskip

For every $d\geq0$ there is a symmetric monoidal functor $\mathcal F:dCob\to Cosp$ which sends a differentiable manifold to its 
underlying topological space and a cobordism to the corresponding cospan of spaces. The functoriality of the assignement follows 
from the fact that inclusions of boundary components in a differentiable manifold are cofibrations (which itself follows from the 
existence of collars), and guaranties that the ordinary push-out is a homotopy push-out. 

\subsubsection{A category of correspondences}

Let $Corr$ be the category with objects being derived Artin stacks and ${\rm Hom}_{Corr}({\bf X},{\bf Y})$ being weak equivalence classes 
of correspondences ${\bf X}\leftarrow {\bf F}\rightarrow {\bf Y}$. Composition of morphisms is given by the homotopy fiber product: 
for ${\bf X}\leftarrow {\bf F}\rightarrow {\bf Y}$ and ${\bf Y}\leftarrow {\bf G}\rightarrow {\bf Z}$, we define 
${\bf X}\leftarrow {\bf G}\circ {\bf F}\rightarrow {\bf Z}$ to be 
${\bf X}\leftarrow {\bf F}\overset{h}{\underset{Y}{\times}}{\bf G}\rightarrow {\bf Z}$. It is symmetric monoidal, 
with monoidal product the Cartesian product. 

\medskip

For any derived Artin stack ${\bf Y}$ we have a symmetric monoidal functor ${\bf Map}\big((-)_B,{\bf Y}\big):Cosp\to Corr$. 
Composed with $\mathcal F$ this gives, for every $d\geq0$, a $d$-dimensional TFT $\mathcal Z_{\bf Y}$ with values in $Corr$. 

\subsection{Semi-classical TFTs from mapping stacks with $n$-symplectic target}\label{ss-TFTlag}

\subsubsection{A category of $d$-oriented cospans}

Let us notice that a compact topological space $\mathbb{X}$ is an oriented Poincar\'e duality space of dimension $d$ (we say 
{\it $d$-oriented}) if and only if the corresponding $\mathcal O$-compact derived Artin stack $\mathbb{X}_B$ is $d$-oriented. 
Hence, if $\mathbb{X}$ is a $d$-oriented compact space then a {\it relative $d$-orientation} on a 
map $f:\mathbb{X}\to \mathbb{X}'$ is defined as a non-degenerate boundary structure on $f_B:\mathbb{X}_B\to \mathbb{X}'_B$. 
\begin{Rem}
Let us give an alternative (and probably easier) description of relative $d$-orientations. We have a fundamental class 
$[\mathbb{X}]:C^*_{sing}(\mathbb{X},\kk)\to\kk[-d]$ and a path from $f_*[\mathbb{X}]$ to $0$ in 
$\map\left(C^*_{sing}(\mathbb{X},\kk),\kk[-d]\right)$. Such a path provides, for any perfect local system $E$ on $\mathbb{X}'$, 
a homotopy commutativity data for the following diagram: 
\begin{equation}\label{eq-reldor}
\vcenter{\xymatrix{
C^*_{sing}(\mathbb{X}',E) \ar[r]\ar[rrd]_0 & C^*_{sing}(\mathbb{X},f^*E)\ar[r]^{\cap[\mathbb{X}]~~~} 
& C_*^{sing}(\mathbb{X},f^*E^\vee)[-d] \ar[d]\\
 && C_*^{sing}(\mathbb{X}',E^\vee)[-d] }}
\end{equation}
Hence the horizontal composition lifts to a map $C^*_{sing}(\mathbb{X}',E)\longrightarrow C^{sing}_*(f,E^\vee)[-d-1]$ denoted $\cap[\mathbb{X}']$, 
where $C_*^{sing}(f,E^\vee)$ is the homotopy cofiber of $C_*^{sing}(\mathbb{X},f^*E^\vee)\to C_*^{sing}(\mathbb{X}',E^\vee)$. 
The non-degeneracy condition can then be restated as follows: $\cap[\mathbb{X}']$ is an isomorphism in $hCpx$. 
\end{Rem}
\begin{Exa}
Oriented Poincar\'e $d$-pairs provide examples of relative $d$-orientations. 
\end{Exa}
We now prove a result that will allow us to define composition of cospans between $d$-oriented spaces. 
\begin{Thm}\label{thm-cospans}
Let $\mathbb{X},\mathbb{Y},\mathbb{Z}$ be three $d$-oriented spaces and assume we have relative $d$-orientations on the maps 
$f:\mathbb{X}\coprod \overline{\mathbb{Y}} \to \mathbb{F}$ and $g:\mathbb{Y}\coprod \overline{\mathbb{Z}} \to \mathbb{G}$, 
where $\overline{\mathbb{Y}}$ means that we consider $\mathbb{Y}$ equipped with the opposite fundamental class. 
Then we have a relative $d$-orientation on 
$t:\mathbb{X}\coprod\overline{\mathbb{Z}}\to\mathbb{T}:=\mathbb{F}\overset{h}{\underset{\mathbb{Y}}{\coprod}}\mathbb{G}$. 
\end{Thm}
\begin{proof}
Let us write $i$, resp.~$j$, for the map $\mathbb{F}\to\mathbb{T}$, resp.~$\mathbb{G}\to\mathbb{T}$. 
We first observe that the relative $d$-orientation on $f$ produces a path from $f_*[\mathbb{X}]$ to $f_*[\mathbb{Y}]$, 
and thus from $i_*f_*[\mathbb{X}]$ to $i_*f_*[\mathbb{Y}]$. Similarly we have a path from $j_*g_*[\mathbb{Y}]$ to $j_*g_*[\mathbb{Z}]$. 
Finally, there is (by definition of the push-out) a path from $i_*f_*[\mathbb{Y}]$ to $j_*g_*[\mathbb{Y}]$. 
Composing these three paths we get a path from $i_*f_*[\mathbb{X}]$ to $j_*g_*[\mathbb{Z}]$, which defines a boundary structure 
on the morphism $t_B$. 

We then have to prove that it is non-degenerate. One way of doing that is by contemplating the following homotopy commuting diagram, 
in which all rows are exact (and $E$ is a perfect local system): 
$$
\vcenter{\xymatrix{
C^*_{sing}(\mathbb{T},E)\ar[r]\ar[d]^{\cap[\mathbb{T}]} & 
C^*_{sing}(\mathbb{F},i^*E)\oplus C^*_{sing}(\mathbb{G},j^*E)\ar[r]\ar[d]^{\cap([\mathbb{F}]+[\mathbb{G}])} & 
C^*_{sing}(\mathbb{Y},f^*i^*E)\ar[d]^{\cap[\mathbb{Y}]}\ar[r]^{~~~~~~~+1} &  \\
C_*^{sing}(t,E^\vee)[-d-1]\ar[r] & \left(C_*^{sing}(f,i^*E^\vee)\oplus C_*^{sing}(g,j^*E^\vee)\right)[-d-1] \ar[r]&
C^*_{sing}(\mathbb{Y},f^*i^*E^\vee)[-d]\ar[r]^{~~~~~~~~~~~~~+1} & 
}}
$$
The first vertical arrow is an isomorphism in $hCpx$ because the last two ones are. 
\end{proof}

Let $Cosp^{d\textrm{-}or}$ be the category with objects being $d$-oriented spaces and 
${\rm Hom}_{Cosp^{d\textrm{-}or}}(\mathbb{X},\mathbb{Y})$ being weak equivalence classes of $d$-oriented maps 
$\mathbb{X}\coprod\overline{\mathbb{Y}}\to\mathbb{F}$, which we call $d$-oriented cospans. 
The composition of two morphisms is given by their homotopy push-out, as in \S\ref{sss-cospans}, with $d$-orientation given by 
Theorem \ref{thm-cospans}. It is again symmetric monoidal, with monoidal product $\coprod$. 

For every $d\geq0$ there is a symmetric monoidal functor $\mathcal F^{or}:dCob^{or}\to Cosp^{d\textrm{-}or}$ which sends an oriented 
$d$-dimensional differentiable manifold to its underlying $d$-oriented space and an oriented cobordism of dimension $d+1$ to the 
corresponding $d$-oriented cospan. It fits into a commuting diagram 
$$
\xymatrix{
dCob^{or}\ar[d]\ar[r]^{\mathcal F^{or}} & Cosp^{d\textrm{-}or}\ar[d] \\
dCob\ar[r]^{\mathcal F} & Cosp
}
$$
where the vertical arrows consist in forgetting orientations. 

\subsubsection{A category of Lagrangian correspondences}\label{sss-lagcorr}

We start with an analog of Theorem \ref{thm-cospans} for Lagrangian structures, which generalizes Theorem \ref{thm-derint} 
on derived Lagrangian intersections. 
\begin{Thm}\label{thm-lagcorr}
Let ${\bf L}_1\overset{f_1}{\longrightarrow}{\bf X}\times {\bf Y}$ and ${\bf L}_2\overset{f_2}{\longrightarrow}{\bf Y}\times{\bf Z}$ be
morphisms of derived Artin stacks, and assume that ${\bf X}$, ${\bf Y}$ and ${\bf Z}$ are equipped with $n$-symplectic structures 
$\omega_{\bf X}$, $\omega_{\bf Y}$ and $\omega_{\bf Z}$. Then there is a map
$$
\mathbb{L}{\rm agr}(f_1,\pi_X^*\omega_{\bf X}-\pi_Y^*\omega_{\bf Y})\times\mathbb{L}{\rm agr}(f_2,\pi_Y^*\omega_{\bf Y}-\pi_Z^*\omega_{\bf Z})
\longrightarrow \mathbb{L}{\rm agr}(g,\pi_X^*\omega_{\bf X}-\pi_Z^*\omega_{\bf Z})\,,
$$
where $g:{\bf L}_{12}:={\bf L}_1\overset{h}{\underset{\bf Y}{\times}}{\bf L}_2\to {\bf X}\times{\bf Z}$. 
\end{Thm}
\begin{proof}
The proof is very similar to the one of Theorem \ref{thm-cospans}. 
We first observe that the result we want to prove holds for isotropic structures. Namely, a Lagrangian structure on $f_1$ is the same as 
a path from $f_1^*\pi_X^*\omega_X$ to $f_1^*\pi_Y^*\omega_Y$ in $\Omega^{2,cl}_{{\bf L}_1}$, a Lagrangian structure on $f_2$ is 
the same as a path from $f_2^*\pi_Y^*\omega_Y$ to $f_2^*\pi_Z^*\omega_Z$ in $\Omega^{2,cl}_{{\bf L}_2}$, and we have a natural path 
from $\pi_{L_1}^*f_1^*\pi_Y^*\omega_{\bf Y}$ to $\pi_{L_2}^*f_2^*\pi_Y^*\omega_{\bf Y}$ in $\Omega^{2,cl}_{{\bf L}_{12}}$. 
Hence we get a path from $\pi_{L_1}^*f_1^*\pi_X^*\omega_X$ to $\pi_{L_2}^*f_2^*\pi_Z^*\omega_Z$ in 
$\Omega^{2,cl}_{{\bf L}_{12}}$, which defines an isotropic structure on $g$. 

We then have to prove that it is non-degenerate. As before, we invite the reader to contemplate a homotopy commuting diagram 
in which all rows are exact: \\
$
\vcenter{\xymatrix{
~~ & ~~ & \mathbb{T}_{{\bf L}_{12}}\ar[r]\ar[d] & \mathbb{T}_{{\bf L}_{1}}\oplus \mathbb{T}_{{\bf L}_{2}}\ar[r]\ar[d]^{\wr} & 
\mathbb{T}_{{\bf Y}}\ar[d]^{\wr}\ar[r]^{~~~+1} &  \\
~~ & ~~ & \mathbb{L}_g[n-1]\ar[r] & (\mathbb{L}_{f_1}\oplus\mathbb{L}_{f_2})[n-1] \ar[r]&
\mathbb{L}_{{\bf Y}}[n]\ar[r]^{~~~+1} & 
}}
$
%The first vertical arrow is an isomorphism in $hCpx$ because the last two ones are. 
\end{proof}
Let $LagCorr_n$ be the category with objects being $n$-symplectic derived Artin stacks and ${\rm Hom}_{LagCorr_n}({\bf X},{\bf Y})$ 
being weak equivalence classes of Lagrangian morphisms ${\bf L}\to{\bf X}\times\overline{{\bf Y}}$, where $\overline{{\bf Y}}$ is 
${\bf Y}$ endowed with the opposite symplectic structure. The composition of two morphisms is their homotopy fiber product, as in 
\S\ref{sss-lagcorr}, endowed with the Lagrangian structure given by Theorem \ref{thm-lagcorr}. 

It again has a symmetric monoidal structure, with monoidal product the Cartesian product. 

\begin{Exa}
Let $f:{\bf X}\to {\bf Y}$ be a morphism between derived $n$-symplectic stacks equiped with a symplectomorphism structure, that is to say a 
Lagrangian structure on the graph ${\bf X}\to {\bf X}\times{\bf Y}$, where ${\bf X}\times{\bf Y}$ is equipped with the difference of the symplectic structures on 
${\bf X}$ and ${\bf Y}$ (i.e.~a non-degenerate path from $f^*\omega_Y$ to $\omega_X$ in $\mathbb{S}{\rm ymp}(X,n)$).  
\begin{Prop}
For any morphism $g:{\bf L}\to {\bf Y}$ we have a map 
$$
\mathbb{L}{\rm agr}(g,\omega_Y)\longrightarrow\mathbb{L}{\rm agr}(\pi_1,\omega_X),
$$
where $\pi_1:{\bf Z}:={\bf X}\overset{h}{\underset{{\bf Y}}{\times}}{\bf L}\to{\bf X}$ is the ``first projection''. 
\end{Prop}
\begin{proof}
This is a direct corollary of Theorem \ref{thm-lagcorr}. 
\end{proof}
\end{Exa}
\begin{Exa}
Let $H,G$ be reductive algebraic groups and let $X$ be a quasi-Hamiltonian $G^{\times2}\times H$-space: there is an equivariant map 
$X\to G^{\times2}\times H$ such that $[X/G^{\times2}\times H]\to [G^{\times2}\times H/G^{\times2}\times H]$ has a Lagrangian structure. 
Observe that the diagonal map $G\to G\times G$ gives rise to two morphisms 
$$
[G\times G/G]\to [G/G]\quad\textrm{and}\quad[G\times G/G]\to [G\times G/G\times G]\,.
$$
One can prove that their product $[G\times G/G]\longrightarrow [G/G]\times \overline{[G\times G/G\times G]}$ carries a Lagrangian structure, 
and thus so does $[G^2\times H/G\times H]\longrightarrow [G\times H/G\times H]\times \overline{[G^2\times H/G^2\times H]}$. 

From Theorem \ref{thm-lagcorr} we get that 
$$
[X/G\times H]=[G^2\times H/G^2\times H]\underset{[G^2\times H/G\times H]}{\overset{h}{\times}}[X/G^2\times H]\longrightarrow [G\times H/G\times H]
$$
carries a Lagrangian structure. Following \cite[Section 6]{AMM} we call it the {\it internal fusion} of $X$, which consists of 
a quasi-Hamiltonian structure on the composed map $X\to G^2\times H\to G\times H$. This gives a new interpretation of \cite[Theorem 6.1]{AMM}. 
We refer to \cite{Safr} for more details about the relevance of derived symplectic and Lagrangian structures for quasi-Hamiltonian geometry. 
\end{Exa}

\subsubsection{Oriented TFTs with values in Lagrangian correspondences}

We fix an $n$-symplectic stack $({\bf Y},\omega)$. 
\begin{Thm}\label{thm-tft}
The symmetric monoidal functor ${\bf Map}\big((-)_B,{\bf Y}\big):Cosp\to Corr$ can be lifted to a symmetric monoidal functor 
$Cosp^{d\textrm{-}or}\to LagCorr_{n-d}$. 
\end{Thm}
As a consequence we get that the functor $\mathcal Z_{\bf Y}:dCob\to Corr$ can be lifted to a symmetric monoidal functor 
$\mathcal Z_{\bf Y}^{or}:dCob^{or}\to LagCorr_{n-d}$, which means that we have a commuting diagram
$$
\xymatrix{
dCob^{or}\ar@/^2pc/[rrr]^{\mathcal Z_{\bf Y}^{or}}\ar[r]^{\mathcal F^{or}}\ar[d] & Cosp^{d\textrm{-}or}\ar[d]\ar[rr] && \ar[d]LagCorr_{n-d} \\
dCob\ar@/_1pc/[rrr]_{\mathcal Z_{\bf Y}}\ar[r]^{\mathcal F} & Cosp \ar[rr]^{{\bf Map}\big((-)_B,{\bf Y}\big)} && Corr
}
$$
\begin{proof}[Proof of Theorem \ref{thm-tft}]
On objects, if $\mathbb{X}$ is $d$-oriented then by Theorem \ref{thm-mapping} we have an $(n-d)$-symplectic structure on the mapping stack 
${\bf Map}\big((\mathbb{X})_B,{\bf Y}\big)$. On morphisms, if $\mathbb{X}\coprod\overline{\mathbb{Y}}\to\mathbb{F}$ carries a relative 
$d$-orientation then by Theorem \ref{thm-maplag} the morphism 
$$
{\bf Map}\big((\mathbb{F})_B,{\bf Y}\big)\longrightarrow{\bf Map}\big((\mathbb{X}\coprod\overline{\mathbb{Y}})_B,{\bf Y}\big)=
{\bf Map}\big((\mathbb{X})_B,{\bf Y}\big)\times\overline{{\bf Map}\big((\mathbb{Y})_B,{\bf Y}\big)}
$$
inherits a Lagrangian structure. It is very easy to check from the construction that this is compatible with composition of morphisms. 
\end{proof}

%%%%%%%%%%%%%%%%%%%%%%%%%%%%%

\section*{Concluding remarks}
\addcontentsline{toc}{section}{Concluding remarks}

In a subsequent work we will show how the above constructions can lead to fully extended oriented TFTs in the sense of Baez-Dolan and Lurie 
(see \cite{BD,Lu1}), as well as to theorie with boundary and/or defects. We aslo conjecture that our approach could provide a rigorous 
construction of the $2$ dimensional TFT with values in holomorphic symplectic manifolds that has been discovered by Moore-Tachikawa \cite{MT}. 

Below we sketch the construction of semi-classical TFTs with boundary conditions from relative derived mapping stacks (see also \cite{Cal-survey}). 

\subsection*{TFTs with boundary conditions}

We briefly explain how can one extend our methods to the construction of TFTs in the presence of boundary conditions. 

\subsubsection*{The category of cobordisms with boundary}

One can introduce the category $dCob^{(or)}_{bnd}$ with objects being (oriented) compact $d$-dimensional manifolds with boundary, 
and morphisms being diffeomorphisms classes of (oriented) $(d+1)$-dimensional cobordisms with boundary. The disjoint union turns it into 
a symmetric monoidal category. An {\it (oriented) topological field theory with boundary} is a symmetric monoidal functor 
$dCob^{(or)}_{bnd}\to \mathcal C$. 

\subsubsection*{An oriented TFT with boundary from relative derived mapping stacks}

If we fix an $n$-symplectic stack and a Lagrangian morphism $f:{\bf L}\to{\bf Y}$ then we can construct a symmetric monoidal functor 
$\mathcal Z_f^{or}:dCob^{or}_{bnd}\longrightarrow LagCorr_{n-d}$ which does the following on objects: if $\Sigma$ is an oriented 
compact $d$-dimensional manifolds with boundary then the inclusion map $\iota_\Sigma:\partial\Sigma\to\Sigma$ is relatively $(d-1)$-oriented 
and we define $\mathcal Z_f^{or}(\Sigma):=\Map\big((\iota_\Sigma)_B,f\big)$, which naturally carries an $(n-d)$-symplectic structure by 
Theorem \ref{thm-2.13}. 

One way to prove this is first to construct the following three oriented TFTs with boundary: \\[-0.5cm]
\begin{itemize}
\item $\mathcal Z_1:\Sigma\mapsto\Map\big(\Sigma_B,f\big)$, which takes values in a category $\mathcal C$ where objects themselves are 
Lagrangian morphisms (and morphisms are weak equivalence classes of correspondences between them). \\[-0.6cm]
\item $\mathcal Z_2:\Sigma\mapsto\Map\big((\iota_\Sigma)_B,{\bf Y}\big)$, which also takes values in $\mathcal C$. \\[-0.6cm]
\item $\mathcal Z_3:\Sigma\mapsto\Map\big(\big(\partial\Sigma)_B,{\bf Y})$, which takes values in $LagCorr_{n-d+1}$. \\[-0.5cm]
\end{itemize}
Then one observes that there is a forgetful functor $\mathcal F:\mathcal C\to LagCorr_{n-d+1}$ and that 
$\mathcal F\circ\mathcal Z_1=\mathcal Z_3=\mathcal F\circ\mathcal Z_2$. Moreover, the derived fiber product of Lagrangian morphisms 
gives a functor 
$$
\mathcal G:\mathcal C\!\!\!\underset{LagCorr_{n-d+1}}{\times}\!\!\!\mathcal C\to LagCorr_{n-d}\,,
$$
so that in the end $\mathcal Z_f^{or}=\mathcal G\circ\big(\mathcal Z_1\times\mathcal Z_2\big)$. 

\medskip

The best way to make sense of this is to work with the higher categories of (iterated) Lagrangian correspondences that have been informally 
introduced at the beginning of the paper. 

%One can first consider the category $RelCosp$ of relative cospans: objects are morphisms of compact spaces and 
%${\rm Hom}(\mathbb{X}\to\mathbb{X}',\mathbb{Y}\to\mathbb{Y}')$ are weak equivalence classes of diagrams 
%$$
%\xymatrix{
%\mathbb{X}\ar[r]\ar[d] & \mathbb{F}\ar[d] & \ar[l]\mathbb{Y}\ar[d] \\
%\mathbb{X}'\ar[r] & \mathbb{F}' & \ar[l]\mathbb{Y}'
%}
%$$
%which we can compose using homotopy push-outs as in \S\ref{sss-cospans}. It is symmetric monoidal with monoidal product the disjoint union, 
%and there is a symmetric monoidal functor $dCob_{bnd}\to RelCosp$ ($d\geq0$). 

\newcommand{\bysame}{\leavevmode\hbox to3em{\hrulefill}\,}

\end{document}